\documentclass[colorlinks=true, linkcolor=blue, citecolor=blue, urlcolor=blue]{article}
\usepackage{orcidlink}
\usepackage{booktabs}

\usepackage{doi}
\usepackage[a4paper, top=1.2in, bottom=1.5in, left=1.3in, right=1.3in]{geometry}
\usepackage{boxedminipage}

\PassOptionsToPackage{colorlinks=true, linkcolor=blue, citecolor=blue, urlcolor=blue}{hyperref}

\usepackage{epsfig,amssymb,amsmath,version,amssymb,version,graphicx,fancybox,mathrsfs,bm}
\usepackage{amsthm}

\usepackage{bm}

\usepackage{graphicx}
\usepackage{subcaption}  
\usepackage{cleveref}
\usepackage{tikz}
\usetikzlibrary{arrows.meta, positioning}

\usepackage[numbers,sort&compress]{natbib}

\theoremstyle{plain}
\newtheorem{theorem}{Theorem}[section]
\newtheorem{lemma}[theorem]{Lemma}
\newtheorem{corollary}[theorem]{Corollary}
\newtheorem{proposition}[theorem]{Proposition}

\theoremstyle{definition}
\newtheorem{definition}[theorem]{Definition}

\theoremstyle{remark}
\newtheorem{remark}{Remark}

\usepackage{algorithm}
\usepackage{algpseudocode}

\usepackage{natbib}

\usepackage{orcidlink}


\usepackage{tikz}
\usetikzlibrary{positioning}
\usetikzlibrary{chains}
\usepackage{tikz-cd}
\usepackage{tikz}
\usepackage{pgfplots}
\pgfplotsset{compat=1.18}
\usetikzlibrary{calc,intersections}

\usetikzlibrary{shapes.geometric, arrows, positioning}

\newcommand{\RR}{\mathbb{R}}
\newcommand{\eps}{\varepsilon}
\newcommand{\mcs}{\mathcal{S}}
\newcommand{\ul}{\boldsymbol}
\newcommand{\ml}{\boldsymbol}

\newcommand{\norm}[1]{\left\lVert#1\right\rVert}

\allowdisplaybreaks[4] 

\usepackage{amsmath}
\usepackage{amssymb}

\usepackage{amsfonts}
\usepackage{contour}
\newtheorem{assumption}{Assumption}

\usepackage{orcidlink}
 \allowdisplaybreaks[4] 

 \usepackage{makecell}

\title{Model-Driven Subspaces for Large-Scale Optimization with Local Approximation Strategy} 
\author{Yitong He\orcidlink{0009-0000-1906-3528}\footnote{Department of Mathematics, University of California, Berkeley, \texttt{yitongh25@gmail.com}} \ and Pengcheng Xie\orcidlink{0000-0001-5973-1535}\footnote{Correpsonding author. Applied Mathematics and Computational Research Division,\\
    Lawrence Berkeley National Laboratory, Berkeley, CA 94720, USA. \texttt{pxie@lbl.gov}, \texttt{pxie98@gmail.com}}}

\begin{document}

\maketitle

\begin{abstract}
Solving large-scale optimization problems is a bottleneck and is very important for machine learning and multiple kinds of scientific problems. Subspace-based methods using the local approximation strategy are one of the most important methods. This paper discusses different and novel kinds of advanced subspaces for such methods and presents a new algorithm with such subspaces, called MD-LAMBO. Theoretical analysis including the subspaces' properties, sufficient function value decrease, and global convergence is given for the new algorithm. The related model construction on the subspaces is given under derivative-free settings. In numerical results, performance profiles, and truncated Newton step errors of MD-LAMBO using different model-driven subspaces are provided, which show subspace-dependent numerical differences and advantages of our methods and subspaces.
\end{abstract}


\section{Introduction}

Solving large-scale optimization problems is a fundamental challenge in machine learning and scientific computing. To alleviate the computational burden of full-space optimization, subspace-based methods have gained increasing attention. These methods restrict the search direction to a carefully constructed low-dimensional subspace, thereby reducing per-iteration complexity and improving efficiency. However, this simplification shifts the challenge toward designing subspaces that not only retain key descent information but also ensure theoretical convergence. 
Most existing algorithms employ subspaces spanned by gradients \cite{Yuan_Stoer_1995,zhang2023drsomdimensionreducedsecondorder,zhang2025scalablederivativefreeoptimizationalgorithms}, historical directions \cite{Arnold2023,xie2024newtwodimensionalmodelbasedsubspace}, quasi-Newton updates \cite{gower2019rsnrandomizedsubspacenewton,cartis2022randomisedsubspacegaussnewtonmethod} as well as other recent methods \cite{1artis2025randomsubspacecubicregularizationmethods,1chen2024q,1hare2025expected,1tansley2025scalablesecondorderoptimizationalgorithms,1You_2016_CVPR,dzahini2025noiseawarescalablesubspaceclassical, Cartis2023}. Though effective, such constructions may be insufficient when dealing with nonlinearity or derivative-free scenarios, and only very few of them have theoretical convergence guarantees. In this work, we focus on two-dimensional subspaces as a tractable yet expressive choice. We further enhance subspace design by introducing two key ideas: building subspaces with model gradients and using cubic regularization. The first idea is inspired by the observation that model-driven subspaces typically appear as linear subspaces truncated by certain regions. For this reason, we also call them truncated subspaces, where the truncations are observed to help save the time of solving subproblems. Our motivations of building subspaces with cubic-regularization-based models come from the following aspects. Primarily, though quadratic models are widely used, their gradients only include limited directions compared to the quadratic model with an additional cubic term. Such newly included directions can include previous iterating directions, which may be overlooked by pure quadratic models. Meanwhile, simple observations show that this cubic model preserves the ability to control the model-driven subspaces' dimension, mainly through carefully selecting the domain of model gradients used for generating subspaces. 

Our contributions are threefold. First, we propose a novel subspace construction framework applicable to both gradient-based and derivative-free settings (see \cite{xie2023dfoto,xie2025remuregionalminimalupdating,xieh2,xieyuannew,Conn2009,Alarie2020} for more details on related trust-region based derivative-free optimization). Second, we provide a theoretical analysis showing that our method ensures function value reduction under mild assumptions. Third, we present extensive numerical experiments demonstrating that different subspace constructions can have a substantial impact on convergence behavior. These results highlight the importance of subspace design in modern optimization algorithms.

The following algorithm presents a general framework of the subspace method for large-scale optimization by \citet{Conn_Toint_Sartenaer_Gould_1996}. Such method solves the optimization problem iteratively, while in each iteration, a low-dimensional subspace is generated on which an approximated solution $\bm x_{k+1}$ is obtained through solving the subproblem. It should be noted that $\bm x_{k+1}$ is not required to be the unique global minimizer though writing ``argmin". In the case where $\mathcal{S}_k$ is always 1-dimensional subspace, this method reduces to a line search method. Besides the subspace method, the trust region technique, being introduced in Powell's \cite{RN53} and Mor\'{e}'s \cite{More1980} papers in the 1970s and 1980s of the last century, is also a classic method for optimization. In each iteration, the trust-region algorithm yields an approximate solution close to the current iteration point. The region where the solution lives is called the trust region, where it was originally named ``restricted step methods" in \citet{FletcherBook2000} and then well-known as ``trust region" named by \citet{Sorensen1982}.

\begin{algorithm}
\caption{Subspace-based optimization method framework}
\vskip6pt
\begin{algorithmic}[1]
    \State Input $f: \RR^n \to \RR$, $\bm x_0\in\RR^n$, tolerance $\eps$, $k=0$.
    \For{$\Vert \bm g_k\Vert\geq \eps$}
        \State Choose a subspace $\mathcal{S}_k \subset\RR^n$.
            \State Calculate $\bm x_{k+1} \approx \mathrm{argmin} \{f(\bm x) \mathrel{:} \bm x\in \bm x_k+\mathcal{S}_k\}$.
        \State Let $k\leftarrow k+1$.
    \EndFor
\end{algorithmic}
\end{algorithm}

\section{Model-driven subspaces}
In this section, we introduce the truncated subspace, a novel class of subspaces that is characterized by being model-driven. Algorithm \ref{algo: complete algorithm} presents a general framework of subspace model-based (derivative-based) optimization in which our truncated subspace is adopted. For the derivative-free case, we provide a class of models as an alternative and prove their efficiency in Section \ref{section: Model functions for derivative-free algorithm}. In Section \ref{section: subspace and dimension}, the definition of truncated subspaces as well as their dimensionality is discussed. In Section \ref{section: Trust region subproblem on subspace}, we show how the classic truncated trust-region method is modified to solve the trust-region subproblem in truncated subspaces. For most of our truncated subspaces, a Cauchy-type sufficient decrease can be achieved by our solution.
To define the truncated subspace, we first introduce the following map $\mathcal{F}$, where 
    \begin{align}
        \mathcal{F}&:\{\text{models}\}\to\{\text{subspaces}\},\ p_k(\bm{s})\mapsto \bigcup\limits_{\bm{s}\in\mathcal{W}}{\rm span}\{\nabla p_k(\bm{s})\},\text{ or }\label{eq:functor 1}\\
        \mathcal{F}&:\{\text{models}\}\to\{\text{subspaces}\},\ p_k(\bm{s})\mapsto {\rm span}\{\bigcup\limits_{\bm{s}\in\mathcal{W}}\nabla p_k(\bm{s})\},\label{eq:functor 2}
    \end{align}
where $\mathcal{W}$ is a proper set of vectors. The subspace generated by \eqref{eq:functor 1} is called the first type of subspace, and the subspace generated by \eqref{eq:functor 2} is called the second type of subspace.
We will show in this section that different choices of $\mathcal{W}$ may yield different subspaces. Then, $\mathcal{F}$ corresponds each model $p_k$ to a subspace $\mathcal{S}=\mathcal{F}(p_k)$ which consists of directions provided by gradients of $p_k$.

A simple observation shows that, compared to subspaces generated by map \eqref{eq:functor 2}, which includes every linear combination of function gradients of $p_k$, subspaces generated by map \eqref{eq:functor 1} seem only to include limited directions of function gradients and
miss a lot of directions of their linear combination. However, we will show in the following sections that most of these truncated subspaces, though they may fail to span the whole linear subspace, do not exclude the minimizer of the projected model $m_k$ in the projected low-dimensional subspace. Hence, such limited subspace processes the advantages of both saving the searching area and ensuring the function value reduction. 
The following Algorithm \ref{algo: complete algorithm} provides the framework of {M}odel-{D}riven {L}ocal {A}pproximation {M}odel-{B}ased {O}ptimization (MD-LAMBO) where the subspaces are generated by \eqref{eq:functor 1} and \eqref{eq:functor 2}. Notice that parameters $\sigma$ and $\sigma_k$ are independent of each other, where $\sigma$ (allowed to be zero, kept unchanged through iterations) is only adopted in subspace-constructing while $\sigma_k$ (adapted at each iteration) is involved in subproblem-solving. The reason that we keep using the same $\sigma$ value to build our subspace is, it can be deduced from Theorem \ref{thm: sp. gene. by cup+span} and \ref{thm:sp. gene. by span+cup} in Subsection \ref{section: subspace and dimension} that, for positive $\sigma$, the mathematical expressions of most of our subspaces will not be affected by the exact value of $\sigma$ at all. 
In the following algorithm, we take the 2-dimensional truncated subspace as an example to explain the framework of MD-LAMBO, where the two-dimensional projection map $\tau^{(k)}_{\bm a,\bm b}$ mentioned in Step 3 is defined by $\tau^{(k)}_{\bm a,\bm b}:\bm x\mapsto((\bm x-\bm x_k)^\top\bm a,(\bm x-\bm x_k)^\top\bm b)$.

\begin{algorithm}
\caption{MD-LAMBO}
    \begin{algorithmic}[1]
   \State \textbf{Input.} $\bm x_0\in\mathbb{R}^n$, object function $f$, $\gamma_2\geq\gamma_1>1$ (or $\gamma_{inc>1}$ and $1>\gamma_{dec}>0$), $1>\eta_2\geq\eta_1>0$, \textcolor{black}{$\sigma\geq 0$} and $\sigma_0>0$. Let $k=0$.
   \State \textbf{Step 1. (Constructing the approximation function)} Obtain the second-order Taylor polynomial of $f$ at $\bm{x}_k$ and let
   \begin{equation*}p_k(\bm{s})=f(\bm{x}_k)+\nabla f(\bm{x}_k)^\top\bm{s}+\frac{1}{2}\bm{s}^\top \nabla^2f(\bm x_k)\bm{s}+\frac{\sigma}{3} \Vert \bm{s}\Vert^3.
   \end{equation*}

    \State \textbf{Step 2. (Generating the subspace)} Use the global approximating function $p_k$ and the functor $\mathcal{F}$ to generate the truncated subspace $\mathcal{S}_k=\mathcal{F}(p_k(\bm{s}))$.

    \State \textbf{Step 3. (Obtaining the projected subspace)} Let $\bm{a}_k=\nabla f(\bm{x}_k)$, obtain $\bm{b}_k\in \mathcal{S}_k$ such that $\bm a_k$ and $\bm b_k$ are linear independent. Apply the Gram-Schmidt process to $\bm a_k$, $\bm b_k$ and get unit vectors $\bm a_k^{(1)},\bm a_k^{(2)}$ that satisfies
    $\langle \bm{a}_k^{(1)},\bm{a}_k^{(2)}\rangle=0$. 
    Obtain the projected subspace $\hat{\mathcal{S}_k}=\tau_{\bm{a}_k^{(1)},\bm{a}_k^{(2)}}^{(k)}(\bm x_k+\mathcal{S}_k)$.
    
    
    \State \textbf{Step 4. (Construct model with adaptive cubic term)} Obtain the model $m_k$ on the projected subspace $\hat{\mathcal{S}_k}$.
    \begin{equation*}
        m_k(x,y)= f(\bm{x}_k)+\bm{P}_k^\top \nabla f(\bm{x}_k) (x,y)^\top+\frac{1}{2}(x,y)\bm{P}_k^\top\nabla^2f(\bm{x}_k)\bm{P}_k(x,y)^\top+\frac{\sigma_k}{3}\Vert (x,y)^\top\Vert^3,
    \end{equation*}
    where $\bm{P}_k=[\bm{a_k,\bm{b}_k}]\in\mathbb{R}^{n\times 2}$ is the projection matrix.
    
     \State \textbf{Step 5. (Trust-region trial step)} Solve the subproblem of $m_k$ on the projected subspace $\hat{\mathcal{S}}_k$ and obtain the solution $(x^*,y^*)$. Then compute
        \begin{equation*}     \rho_k=\frac{f(\bm{x}_k)-f(\bm{x}_k+\bm{P}_k (x^*,y^*)^\top)}{m_k(0,0)-m_k(x^*,y^*)}.
        \end{equation*}        

        \State \textbf{Step 6. Updating} Update and obtain $x_{k+1}$ and $\sigma_k$ based on the \textcolor{black}{ARC} method.
        Increment $k$ by one and go to \textbf{Step 1}.

    \end{algorithmic}
    \label{algo: complete algorithm}
\end{algorithm}
For the $q$-dimensional truncated subspace (satisfying $\nabla f(\bm x_k)\in\mathcal{S}_k$) the projection matrix $\bm P_k=[\bm a_k^{(1)},\cdots,\bm a_k^{(q)}]$ is obtained through applying Gram-Schmidt process to $q$ linearly independent vectors in $\mathcal{S}_k$ where the leftmost vector is $\nabla f(\bm x_k)$. Therefore there always be $\bm a_k^{(1)}=\nabla f(\bm x_k)/\Vert\nabla f(\bm x_k)\Vert$. Moreover, their projection map is (generalization of $\tau_{\bm a_k,\bm b_k}^{(k)}$ in Algorithm \ref{algo: complete algorithm}) \begin{equation*}
    \tau_k: \bm x_k+\mathcal{S}_k\to\mathbb{R}^q,\ \bm x\to ( (\bm x-\bm x_k)^\top\bm a_k^{(1)},\cdots,(\bm x-\bm x_k)^\top\bm a_k^{(q)}),
\end{equation*}
where $q=\dim(\mathcal{S}_k)$. The definition of the dimensionality of truncated subspaces is given by Definition \ref{def:dimension of subspace} in Section \ref{section: subspace and dimension}. It should be noted that for truncated subspaces $\mathcal{S}_k$ which do not satisfy $\nabla f(\bm x_k)\in\mathcal{S}_k$, to obtain the projection matrix, we just select $q$ linearly independent vectors $\bm y_1,\cdots,\bm y_q$, then apply the Gram-Schmidt process. We present sixteen subspaces in Section \ref{section: subspace and dimension} in total. By Theorem \ref{thm: sp. gene. by cup+span} and Theorem \ref{thm:sp. gene. by span+cup} which gives the explicit mathematical expressions of each subspace, only two subspaces $\mathcal{S}_{\bm s_1,\bm s_2}^{(1,0)}$ and $\mathcal{S}_{\bm s_1,\bm s_2}^{(1,\sigma)}$ do not satisfy $\nabla f(\bm x_k)\in \mathcal{S}_k$ therefore needed to be treated specially when obtaining the projection matrix.

\subsection{Subspaces generated by model gradients}\label{section: subspace and dimension}

At each iteration, a proper subspace $\mathcal{S}_k$ is generated utilizing the gradient directions of the function $p_k$. In this subsection, we provide 16 novel subspaces in total, where Table \ref{tab:subspace-1} lists 8 kinds of subspaces generated by \eqref{eq:functor 1} and Table \ref{tab:subspace-2} collects the rest of 8 kinds of subspaces given by \eqref{eq:functor 2}. For convenience, we denote $\bm s_1=({\bm x_{k-2}-\bm x_k})/{\|\bm x_{k-2}-\bm x_k\|}$, $\bm s_2=({\bm x_{k-1}-\bm x_k})/{\|\bm x_{k-1}-\bm x_k\|}$ and without loss of generality, assume $\bm s_1\neq \bm s_2$.

\begin{table}[htbp]
    \caption{The first type of subspaces: $\mathcal{S}_k=\bigcup_{\bm{s}\in \mathcal{W}} {\rm span}\{\nabla p_k(\bm{s})\}$}
    \centering
    \begin{tabular}{m{3cm}<{\centering}m{4cm}<{\centering}m{4cm}<{\centering}}
    \toprule
     $\mathcal{W}$ & $p_k(\bm{s})=Q_k(\bm{s})$ & $p_k(\bm{s})=Q_k(\bm{s})+\sigma \|\bm{s}\|^3$ \\
    \midrule
       $\mathcal{B}(\bm 0;1)$  & $\mathcal{S}_{\mathcal{B}}^{(1,0)}\subseteq C_{\mathcal{B}(\bm g;\Delta_{0})}$  & $\mathcal{S}_{\mathcal{B}}^{(1,\sigma)}\subseteq C_{\mathcal{B}(\bm g;\Delta_{\sigma})}$ \\


        span$\{\nabla p_k(\bm 0)\}$ & $\mathcal{S}_{l}^{(1,0)}\cong\mathbb{R}^2\backslash \mathbb{R}^1$ &   $\mathcal{S}_{l}^{(1,\sigma)}\subseteq \RR^2$\\
     
        \vspace{1mm}
        $\{\bm s_1,\bm s_2\}$ &  \vspace{1mm}$\mathcal{S}_{\bm s_1,\bm s_2}^{(1,0)}\cong\mathbb{R}^1\vee\mathbb{R}^1$ & \vspace{1mm}$\mathcal{S}_{\bm s_1,\bm s_2}^{(1,\sigma)}\cong\mathbb{R}^1\vee\mathbb{R}^1$ \\

        span$\{\bm s_1,\bm s_2\}$ & $\mathcal{S}_{\langle\bm s_1,\bm s_2\rangle}^{(1,0)}\cong\mathbb{R}^3\backslash\mathbb{R}^2$ & $\mathcal{S}_{\langle\bm s_1,\bm s_2\rangle}^{(1,\sigma)}\subseteq\RR^5$
        \\
    \bottomrule
    \end{tabular}
    \label{tab:subspace-1}
\end{table}

\begin{table}[htbp]
    \centering
    \caption{The second type of subspaces: $\mathcal{S}_k=\text{span}\{\bigcup_{\bm{s}\in\mathcal{W}}\nabla p_k(\bm{s})\}$}
    \centering
    \begin{tabular}{m{3cm}<{\centering}m{4cm}<{\centering}m{4cm}<{\centering}}
    \toprule
     $\mathcal{W}$ & $p_k(\bm{s})=Q_k(\bm{s})$  & $p_k(\bm{s})=Q_k(\bm{s})+\sigma \|\bm{s}\|^3$\\
    \midrule
        $\mathcal{B}(\bm 0;1)$  & $\mathcal{S}_{\mathcal{B}}^{(2,0)}=\mathbb{R}^n$ & $\mathcal{S}_{\mathcal{B}}^{(2,\sigma)}=\mathbb{R}^n$  \\
        span$\{\nabla p_k(0)\}$ & $\mathcal{S}_{l}^{(2,0)}\cong\mathbb{R}^q,q\leq 2$ & 
$\mathcal{S}_{l}^{(2,\sigma)}\cong\mathbb{R}^q,q\leq 2$\\
        $\cup_{i=1}^2{\rm span}\{\bm s_i\}$ & 
        $\mathcal{S}_{\langle\bm s_1\rangle,\langle\bm s_2\rangle}^{(2,0)}\cong\mathbb{R}^q,q\leq 3$
        &
        $\mathcal{S}_{\langle\bm s_1\rangle,\langle\bm s_2\rangle}^{(2,\sigma)}\cong\mathbb{R}^q,q\leq5$
        \\
        span$\{\bm s_1,\bm s_2\}$ & $\mathcal{S}_{\langle\bm s_1,\bm s_2\rangle}^{(2,0)}\cong\mathbb{R}^q,q\leq3$ & $\mathcal{S}_{\langle\bm s_1,\bm s_2\rangle}^{(2,\sigma)}\cong\mathbb{R}^q,q\leq5$\\
     \bottomrule   
    \end{tabular}   
    \label{tab:subspace-2}
\end{table}
The first column of Table \ref{tab:subspace-1} and \ref{tab:subspace-2} list four cases of inspiring regions $\mathcal{W}$, where $\mathcal{W}$ takes from two points in the iteration, a line spanned by the directions of function gradient, two lines with directions provided by iterative points, the plane spanned by such two iterative points and the unit ball in $\mathbb{R}^n$. \textcolor{black}{Noted that symbol $\cong$ in both tables means that there is a linear bijection from $\mathcal{S}_k$ to the space following $\cong$}. In Table \ref{tab:subspace-1}, the inclusion symbol $\subseteq$ in the first, second and the last row roughly means the inclusion between subsets of $\RR^n$, where $\mathcal{S}_l^{(1,\sigma)}\subseteq \mathbb{R}^2$ ($\mathcal{S}_{\langle \bm s_1,\bm s_2\rangle}^{(1,\sigma)}\subseteq\RR^5$) in fact means the subspace $\mathcal{S}_l^{(1,\sigma)}\subseteq \mathcal{S}$ ($\mathcal{S}_{\langle \bm s_1,\bm s_2\rangle}^{(1,\sigma)}\subseteq\mathcal{S}$) for some 2-dimensional (5-dimensional) linear subspace of $\RR^n$. For subspace $\mathcal{S}_{\mathcal{B}}^{(1,0)}$, however, we write $\mathcal{S}_{\mathcal{B}}^{(1,0)}\cong \mathcal{B}(\bm 0;1)$ to indicate that the two spaces are homeomorphic to each other under the Euclidean topology. Generally, we use $\mathcal{S}$ to represent the subspace, the subscript to represent the inspiring region, and the superscript to indicate the type and existence of the cubic term. For example, we use $\mathcal{S}_l^{(1,\sigma)}$ to refer to the subspace $\cup_{\bm s\in\mathrm{span}\{\nabla p_k(\bm 0)\}}\mathrm{span}\{\nabla p_k(\bm s)\}$ where $p_k(\bm s)=Q_k(\bm s)+\sigma\Vert \bm s\Vert ^3$. Here the subscript $l$ is used as an indicator of $\mathcal{W}=\mathrm{span}\{\nabla p_k(\bm 0)\}$ since the inspiring region is a one-dimensional linear subspace (a straight line), and $(1,\sigma)$ indicates that the subspace belongs to the first-type while the coefficient of cubic term $\sigma\Vert\bm s\Vert^3$ is nonzero, i.e., the cubic term is adopted. Similarly, $\mathcal{S}_{\langle\bm s_1,\bm s_2\rangle}^{(2,0)}$ stands for the subspace $\mathrm{span}\{\cup_{\bm s\in\mathrm{span}\{\bm s_1,\bm s_2\}}\nabla p_k(\bm s)\}$ where $p_k(\bm s)=Q_k(\bm s)$. Here, the superscript $(2,0)$ means the subspace is second-type while the model $p_k$ is pure quadratic, i.e., the coefficient $\sigma$ is zero.

These inspiring regions are considered for their dimensions vary from one to $n$ when being regarded as a collection of vectors. Since for quadratic functions, calculating gradients can be viewed as a linear transformation with a translation, such inspiring regions tend to yield our desired truncated two-dimensional subspaces.
In the case where the inspiring region is a one-dimensional linear subspace, we select the vector $\nabla p_k(\bm 0)$ as the direction of the line mainly for the following reasons. Firstly, $p_k(\bm s)$ decreases the most intensively along the $-\nabla p_k(\bm s)$, therefore it is important to include directions generated by points on the line spanned by $\nabla p_k$ when constructing subspaces. Secondly, since $p_k$ is the approximation of the object function $f$ whose gradient is only known to be accurate at $\bm s=\bm 0$, we take the direction of $\nabla p_k(\bm 0)$ as representative. 

\begin{table}[htbp]
\caption{Projection of $\bm J=(1,\cdots,1)^\top$ onto different subspaces}
    \centering
    \begin{tabular}{m{1.8cm}<{\centering}m{2cm}<{\centering}m{2cm}<{\centering}m{1.8cm}<{\centering}m{2cm}<{\centering}m{2cm}<{\centering}}
    \toprule
   
      $\mathcal{S}_k$   & $\bm P_k$ & $\tau_k(\bm J)$ & $\mathcal{S}_k$   & $\bm P_k$ & $\tau_k(\bm J)$ \\
    \midrule    $\mcs_{\mathcal{B}}^{(1,\sigma)}$  & $[e_1,\cdots,e_q]$ & $ (0,1,\cdots,1)^\top$ & $\mcs_{\mathcal{B}}^{(2,\sigma)}$  & $[e_1,\cdots,e_q]$ & $(0,1,\cdots,1)^\top$\\
    $\mcs_{l}^{(1,\sigma)}$  & $e_1$ & $0$ & $\mcs_{l}^{(2,\sigma)}$  & $e_1$ & $0$\\
    $\mcs_{\bm s_1,\bm s_2}^{(1,\sigma)}$  & $[e_1,e_2]$ & $(0,1)^\top$ & $\mcs_{\langle\bm s_1\rangle,\langle\bm s_2\rangle}^{(2,\sigma)}$  & $[e_1,e_2]$ & $(0,1)^\top$\\
    $\mcs_{\langle\bm s_1,\bm s_2\rangle}^{(1,\sigma)}$  & $[e_1,e_2]$ & $(0,1)^\top$& $\mcs_{\langle\bm s_1,\bm s_2\rangle}^{(2,\sigma)}$  & $[e_1,e_2]$ & $(0,1)^\top$\\
    \bottomrule
    \end{tabular}
    \label{tab: projection of J onto type-1 sp}
\end{table}

Table \ref{tab: projection of J onto type-1 sp} lists the projection of vector $\bm J=(1,1,\cdots,1)^\top\in\mathbb{R}^n$ onto the 16 subspaces in the case where the object function $f(\bm x)=\Vert \bm x\Vert^2$ ($\tau_k$ can be extended to the whole space $\RR^n$ with the same definition if $\bm J\notin\bm x_k+\mathcal{S}_k$). Let $\bm x_k=e_1=(1,0,\cdots,0)^\top.$ We choose $ \bm s_1=e_1,\bm s_2=e_2$. Note that $\bm P_k$ is the projection matrix whose construction has already been discussed right after the framework of Algorithm \ref{algo: complete algorithm}. By showing the projection matrix and results of a specific vector under concrete settings, we hope to give a clearer picture of the projection map $\tau_k$ and therefore the projected subspaces $\hat{\mathcal{S}_k}$. For the convenience of calculation, we select the linearly independent vectors from $\nabla p_k(\mathcal{W})$ and then apply the Gram-Schmidt process to obtain column vectors of $\bm P_k$. Results in Table \ref{tab: projection of J onto type-1 sp} indicate that, one can always select the base of $\mathcal{S}_k$ to make the projection matrix as simple as the identity matrix. Moreover, under the same $f,\bm x_k,\bm s_1,\bm s_2$ settings, as long as the linearly independent vectors are selected from $\nabla p_k(\mathcal{W})$, $\mathcal{S}_\mathcal{W}^{(i,\sigma)}$ will share the same projection matrix for the same $\mathcal{W}$ and for all $i=1,2,\sigma\geq 0$. Furthermore, one may notice that $\tau_k(\bm J)=0$ in the case $\mathcal{S}_k=\mathcal{S}_l^{(i,\sigma)}$ ($i\in\{1,2\},\sigma\geq 0$), meaning that $\bm J$ (in this case $\bm J\notin\bm x_k+\mathcal{S}_k$) is corresponded to the origin, the same point as $\bm x_k$, on the projected subspace $\hat{\mathcal{S}}_k$. This reflects that, being restricted to the subspace $\mathcal{S}_k$, $\tau_k$ simply functions as an orthogonal projection, which preserves the metric structure within the subspace $\mathcal{S}_k$. However, being extended to the entire space $\RR^n$, $\tau_k$ can compress the points outside and inside the subspace $\mathcal{S}_k$ together and cause the loss of global information.

Throughout the paper, we will use $Q$ or $Q_k$ to represent quadratic functions. Hence in MD-LAMBO, the model $p_k$ can be written as $p_k(\bm s)=Q_k(\bm s)+\sigma\Vert \bm s\Vert^3$ where $\nabla Q_k(\bm 0)=\nabla f(\bm x_k)$, $\nabla^2Q_k(\bm s)=\nabla^2 f(\bm x_k)$ for all $\bm s\in\mathcal{S}_k$. Since $\nabla f(\bm x_k)=\bm 0$ implies the termination of our algorithm, and when $\bm x_k$ is close enough to a minimal point, $\nabla^2 f(\bm x_k)$ is positive semi-definite. Therefore, for a clear depiction of truncated subspaces, we will assume $\nabla Q_k(\bm 0)\neq \bm 0$ and (slightly stronger than) $\nabla^2 Q_k(\bm s)$ is positive definite in the following two theorems. Also, to simplify the arguments about the rank of $\mathcal{S}_k$ when it is viewed as a group of vectors, we may assume that $\nabla Q_k(\bm 0)$ and $\nabla^2 Q_k(\bm s)$ (for all $\bm s$) are linearly independent.

\begin{theorem}\label{thm: sp. gene. by cup+span}
    Let $p_k$ be the model in MD-LAMBO. Then the first type of subspace $\mathcal{F}(p_k)=\bm{x}_k+\bigcup_{\bm{s}\in \mathcal{W}} {\rm span}\{\nabla p_k(\bm{s})\}$ is
    \begin{enumerate}
        \item $\mathcal{S}_{\mathcal{B}}^{(1,\sigma)}\subseteq C_{\mathcal{B}(\bm g;\Delta_\sigma)}(\subseteq\mathbb{R}^n$), if   $\mathcal{W}=\mathcal{B}(\bm 0;1)$;
        
        \item $\mathcal{S}_{\bm s_1,\bm s_2}^{(1,\sigma)}\cong\mathbb{R}^1\vee_{\bm x_k}\mathbb{R}^1$, if $\mathcal{W}=\{\bm s_1,\bm s_2\}$;   \item$\mathcal{S}_{\langle\bm s_1,\bm s_2\rangle}^{(1,0)}
            \begin{cases}     \cong\mathbb{R}^3\backslash\mathbb{R}^2,&\text{ if }\sigma =0\\
                \subseteq{\rm span}\{\bm g,\bm B\bm s_1,\bm B\bm s_2,\bm s_1,\bm s_2\},&\text{ if }\sigma>0
            \end{cases}$, in the case where $\mathcal{W}={\rm span}\{\bm s_1,\bm s_2\}$;   \item$\mathcal{S}_{l}^{(1,\sigma)}\begin{cases}
            \cong\RR^2\backslash\RR^1,&\text{ if }\sigma=0\\
            =\{ar_k^{(1)}+br_k^{(2)}:a,b\in\RR, ab\geq 0\},r_k^{(i)}\in\RR^n,&\text{ if }\sigma>0
        \end{cases}$, in the case where $\mathcal{W}=\mathrm{span}\{\nabla p_k(\bm 0)\}$.
    \end{enumerate}
\end{theorem}
\begin{proof}{Proof}
Suppose $Q_k(\bm{s})=c+\bm{g}^\top \bm{s}+\frac{1}{2}\bm{s}^\top\bm{B}\bm{s}$, then $\nabla p_k(\bm{s})=\bm{g}+(\bm{B}+\sigma \Vert s\Vert\bm{I})\bm{s}$.

\noindent\textbf{1.}
\textbf{(1)} In the case $\sigma=0$, consider the gradient function
restricted to the unit ball $\mathcal{B}$
\begin{equation*}
    \nabla p_k: \mathcal{B}\to \nabla p_k(\mathcal{B }),\ \bm s\mapsto \bm g+\bm B\bm s.
\end{equation*}
By $\bm B$ is invertible, such a map is a homeomorphism.
Moreover,
\begin{equation*}
    \Vert\nabla p_k(\bm s)-\bm g\Vert=\Vert\bm B\bm s\Vert\leq \Vert\bm B\Vert\Vert\bm s\Vert=\Vert
    \bm B\Vert,
\end{equation*}
where $\Vert \bm B\Vert=\sup_{\bm v\in\mathbb{R}^n-\{\bm 0\}}\Vert\bm B\bm v\Vert/\Vert\bm v\Vert$, which is the operator norm of matrix $\bm B$. Let $\Delta_0=\Vert \bm B\Vert$, then $\nabla p_k(\mathcal{B})\subseteq\mathcal{B}(\bm g;\Delta_0)$. Consider the set $T_{\mathcal{B}(\bm g;\Delta_0)}$ of tangent lines of $\mathcal{B}(\bm g;\Delta)$ through the origin. Then $T_{\mathcal{B}(\bm g;\Delta_0)}$ forms a cone $C_{\mathcal{B}(\bm g;\Delta_0)}$ where 
\begin{equation*}
    C_{\mathcal{B}(\bm g;\Delta)}=\bigcup\limits_{\bm y\in\mathcal{B}(\bm g;\Delta_0)}{\rm span}\{\bm y\}.
\end{equation*}
Then $\mathcal{S}_{\mathcal{B}}^{(1,0)}=\cup_{\bm y\in\nabla m_k(\mathcal{B})}{\rm span}\{\bm y\}
\subseteq\cup_{\bm y\in\mathcal{B}(\bm g;\Delta_0)}{\rm span}\{\bm y\}=C_{\mathcal{B}(\bm g;\Delta_0)}$.

\noindent\textbf{(2)} In the case $\sigma>0$, let $\mathcal{S}_t=\{\bm s\in\mathbb{R}^n:\Vert\bm s\Vert=t\}$. Though the cubic term $\sigma\Vert\bm s\Vert^3/3$ makes the transformation matrix $(\bm B+\sigma\Vert\bm s\Vert\bm I)$ vary for different $\bm s$, once we fix $\Vert \bm s\Vert=t$ and consider the gradient function $\nabla p_k$ restricted to the sphere, i.e., $\nabla p_k:\mathcal{S}_t\to \nabla p_k(\mathcal{S}_t)$.
Such a restricted map is still a homeomorphism. So for each $t\in(0,1)$, the space $\nabla p_k(\mathcal{S}_t)$ looks like a sphere centering at $\bm g$. Moreover, we have $\Vert\nabla p_k(\bm s)-\bm g\Vert=\Vert(\bm B+\sigma t\bm I)\bm s\Vert\leq \Vert \bm B+\sigma t\bm I\Vert$. Let $\Delta_{\sigma}=\sup_{t\in[0,1)}\Vert \bm B+\sigma t\bm I\Vert$. Then $\nabla p_k(\mathcal{B})\subseteq \mathcal{B}(\bm g;\Delta_{\sigma})$. Define $C_{\mathcal{B}(\bm g;\Delta_{\sigma})}$ the same way as in \textbf{(1)}, then $\mathcal{S}_{\mathcal{B}}^{(1,\sigma)}\subseteq C_{\mathcal{B}(\bm g;\Delta_{\sigma})}$.

\noindent\textbf{2.} Obviously.

\noindent\textbf{3.} We discuss the two cases $\sigma =0$ and $\sigma>0$ separately. In the case where $\sigma=0$, then $\mathcal{S}_{\langle\bm s_1,\bm s_2\rangle}^{(1,0)}={\rm span}\{\bm g,\bm B\bm s_1,\bm B\bm s_2\}\backslash{\rm span}\{\bm B\bm s_1,\bm B\bm s_2\}$. However, if $\sigma>0$, we have 
\begin{equation*}
\mathcal{S}_{\langle\bm s_1,\bm s_2\rangle}^{(1,\sigma)}=\{a\bm g+b\bm B\bm s_{1}+c\bm B\bm s_2+\sigma\sqrt{(\frac{b}{a})^2+(\frac{c}{a})^2}\left(b\bm s_1+c\bm s_2\right):a,b,c\in\mathbb{R},a\neq 0\}\cup\{\bm O\},
\end{equation*}
which is a subset of $\mathrm{span}\{\bm g,\bm B\bm s_1,\bm B\bm s_2,\bm s_1,\bm s_2\}$.

\noindent\textbf{4.} Similarly, we discuss the two cases separately.

\noindent\textbf{(1)} In the case where $\sigma=0$.
$\nabla p_k(\bm{s})=\nabla p_k(t\bm g)=\bm{g}+t\bm{B}\bm{g}$, $t\in\mathbb{R}$. Then for any $\alpha\bm g+\beta\bm B\bm g\in {\rm span}\{\bm g,\bm B\bm g\}$ where $\alpha\neq 0,\beta\in\mathbb{R}$, one can find $a,b\in\mathbb{R}$ $(a\neq 0)$ such that $\alpha\bm g+\beta\bm B\bm g=a(\bm g+b\bm B\bm g)$, i.e., $\alpha\bm g+\beta\bm B\bm g\in{\rm span}\{\bm g+b\bm B\bm g\}$. Therefore ${\rm span}\{\bm g,\bm B\bm g\}\backslash{\rm span}\{\bm B\bm g\}\subseteq\mathcal{S}_l^{(1,0)}\subseteq{\rm span}\{\bm g,\bm B\bm g\}$. Moreover, since $\bm g=\nabla Q_k(\bm 0)\neq \bm 0$, ${\rm span}\{\bm B\bm g\}\notin\mathcal{S}_l^{(1,0)}$. Hence $\mathcal{S}_l^{(1,0)}= {\rm span}\{\bm g,\bm B\bm g\}\backslash{\rm span}\{\bm B\bm g\}$, implying that $\mathcal{S}_l^{(1,0)}\cong \mathbb{R}^2\backslash\mathbb{R}^1$.

\noindent\textbf{(2)} In the case where $\sigma>0$.
\begin{figure}[htbp]
    \centering
    \includegraphics[width=.8\linewidth]{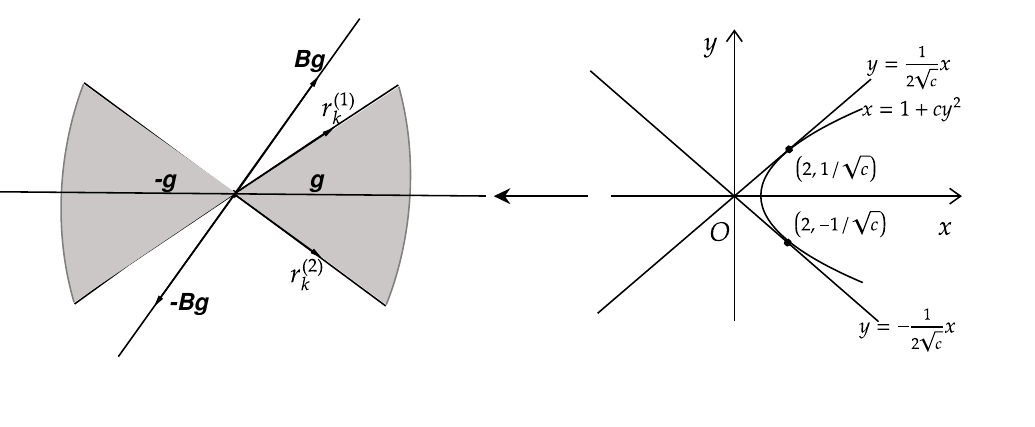}
    \caption{$\mathcal{S}_l^{(1,\sigma)}$ and its corresponding affine plane}
    \label{fig:sector subspace}
\end{figure}
We claim that $\mathcal{S}_l^{(1,\sigma)}$ is the subset of 2-dimensional subspace ${\rm span}\{\bm{g},\bm{Bg}\}$, which is bounded by two lines with the directions $\bm{r}_{k}^{(1)}=2\bm{g}+\bm{Bg}/\sqrt{\sigma \Vert\bm{g}\Vert}$ and $\bm{r}_{k}^{(2)}=2\bm{g}-\bm{Bg}/\sqrt{\sigma \Vert\bm{g}\Vert}$, as illustrated in Figure \ref{fig:sector subspace}.

Write $\bm s\in {\rm span}\{\nabla p_k(\bm 0)\}$ as $\bm s=t\bm g$ where $t\in \mathbb{R}$. Then 
\begin{equation*}\mathcal{S}_l^{(1,\sigma)}=\bigcup\limits_{t\in\mathbb{R}}{\rm span}\{(1+\sigma t^2\Vert \bm g\Vert)\bm g+t\bm{Bg}\}.
\end{equation*}

Establish the affine coordinate system using $\{\bm O,\bm g,\bm B\bm g\}$ as the frame. Let $\bm g$ be the $x$-axis, $\bm B\bm g$ be the $y$-axis and $O$ the origin. By representing each point with its affine coordinates, $\mathcal{S}_l^{(1,\sigma)}$ can be written as
\begin{equation*}
    \mathcal{S}_l^{(1,\sigma)}=\left\{l\left(1+\sigma \Vert \bm g\Vert t^2,t\right):l,t\in\RR\right\}, 
\end{equation*}
which contains the graph of parabola $x=1+cy^2$, where $c=\sigma\Vert\bm g\Vert$. Moreover, each point $(x',y')\in S_l^{(1,\sigma)}$ can be viewed as the intersection of such a parabola and a line through the origin, i.e.,
solution $(x,y)$ of the equation
\begin{equation*}
    \begin{cases}
        x=1+cy^2\\
        y=kx
    \end{cases}
\end{equation*}
for some $k$. Consider equation \eqref{eq: zero dircriminant} of $k$ given by zero discriminant.
\begin{equation}\label{eq: zero dircriminant}
    \frac{1}{c^2k^2}-\frac{4}{c}=0.
\end{equation}
Then the solution $k=1/(2\sqrt{c})$ or $k=-1/(2\sqrt{c})$ gives the expression of the tangents of parabola $y=x/(2\sqrt{c})$ and $y=-x/(2\sqrt{c})$. Let $\bm r_k^{(1)},\bm r_k^{(2)}\in\RR^n$ be the points whose affine coordinates are $(2,1/\sqrt{c})$ and $(2,-1/\sqrt{c})$ respectively. Then we have
\begin{equation*}
    \mathcal{S}_l^{(1,\sigma)}=\{a\bm r_k^{(1)}+b\bm r_k^{(2)}:a,b\in\RR,ab\geq 0\},
\end{equation*}
where $\bm r_k^{(1)}$ and $\bm r_k^{(2)}$ is exactly the same as what we claimed.
\end{proof}

\begin{theorem}\label{thm:sp. gene. by span+cup}
     Let $p_k$ be the model constructed in MD-LAMBO. Then the second type of subspace $\mathcal{F}(p_k)= {\rm span}\{\bigcup_{\bm s\in\mathcal{W}}\nabla p_k(\bm{s})\}$ is

     \begin{enumerate}
         \item $\mathcal{S}_{\mathcal{B}}^{(2,\sigma)}=\mathbb{R}^n$, if $\mathcal{W}=\mathcal{B}(\bm 0;1)\subseteq \mathbb{R}^n$;
         
         \item $\mathcal{S}_{l}^{(2,\sigma)}\cong \mathbb{R}^1$ or $\mathbb{R}^2$, if $\mathcal{W}={\rm span}\{\nabla p_k(\bm 0)\}$.

         \item $\mathcal{S}_{\langle \bm s_1\rangle,\langle\bm s_2\rangle}^{(2,\sigma)}\cong\mathbb{R}^d,\text{ for a }d\leq 5$, if $\mathcal{W}={\rm span}\{\bm s_1\}\cup{\rm span}\{\bm s_2\}$. 
         
         \item $\mathcal{S}_{\langle\bm s_1,\bm s_2\rangle}^{(2,\sigma)}\cong \mathbb{R}^d\text{ for a }d\leq 5$, if $\mathcal{W}={\rm span}\{\bm s_1,\bm s_2\}$.        
     \end{enumerate}
\end{theorem}

\begin{proof}{Proof}
Suppose $Q_k(\bm{s})=c+\bm{g}^\top \bm{s}+\frac{1}{2}\bm{s}^\top\bm{B}\bm{s}$, then $\nabla p_k(\bm{s})=\bm{g}+(\bm{B}+3\sigma \Vert \bm s\Vert\bm{I})\bm{s}$. 

\noindent\textbf{1.} Fix $\sigma$, by $\bm B$ positive definite, for any $t\in[0,1)$, we have $\bm B+3\sigma t\bm I$ invertible. Therefore there must be $n$ linearly independent $\bm s_1,\cdots,\bm s_n$ such that $\Vert \bm s_i\Vert=t\ (i=1,\cdots,n)$. Let $\bm a_i=(\bm B+3\sigma \Vert \bm s_i\Vert \bm I)\bm s_i$, then $\bm a_i$'s are also linearly independent. Since ${\rm span}\{\cup_{\bm s\in\mathcal{B}}\ \left(\bm{g}+(\bm{B}+3\sigma \Vert \bm s\Vert\bm{I})\bm{s}\right)\}={\rm span}\{\bm g,\cup_{\bm s\in\mathcal{B}}(\bm B+3\sigma \Vert \bm s\Vert \bm I)\bm s\}$, $\mathcal{S}_{\mathcal{B}}^{(2,\sigma)}={\rm span}\{\bm g,\cup_{\bm s\in\mathcal{B}}(\bm B+3\sigma \Vert \bm s\Vert \bm I)\bm s\}$. Therefore $\bm a_i\in\mathcal{S}_{\mathcal{B}}^{(2,\sigma)},\ i=1,\cdots,n$. Hence $\mathcal{S}_{\mathcal{B}}^{(2,\sigma)}=\mathbb{R}^n$.

\noindent\textbf{2.} 
Same argument in \textbf{1} implies $\mathcal{S}_{l}^{(2,\sigma)}={\rm span}\{\bm g, (\bm B+3\sigma \Vert t\bm g\Vert\bm I)t\bm g:t\in\mathbb{R}\}$. Simple observations show that $\mathcal{S}_{l}^{(2,0)}={\rm span}\{\bm g,\bm B\bm g \}$ in both $\sigma=0$ and $\sigma>0$ cases.

\noindent\textbf{3.} 
Since $\mathcal{W}={\rm span}\{\bm s_1\}\cup{\rm span}\{\bm s_2\}$, $\mathcal{S}_{\langle \bm s_1\rangle,\langle \bm s_2\rangle}^{(2,\sigma)}={\rm span}\{\bm g,\cup_{\bm s\in\mathcal{W}}(\bm B+3\sigma \Vert \bm s\Vert \bm I)\bm s\}$, in the case where $\sigma=0$,it is easy to see $\mathcal{S}_{\langle \bm s_1\rangle,\langle \bm s_2\rangle}^{(2,0)}={\rm span}\{\bm g,\bm B\bm s_1,\bm B\bm s_2\}$. If $\sigma>0$, otherwise, this implies $\mathcal{S}_{\langle \bm s_1\rangle,\langle \bm s_2\rangle}^{(2,\sigma)}\subseteq{\rm span}\{\bm g,\bm s_1,\bm s_2,\bm B\bm s_1,\bm B\bm s_2\}$. Moreover, since $\bm s_1, \bm s_2\in\mathcal{S}_{\langle \bm s_1\rangle,\langle \bm s_2\rangle}^{(2,\sigma)}$, $\bm B\bm s_1,\bm B\bm s_2\in\mathcal{S}_{\langle \bm s_1\rangle,\langle \bm s_2\rangle}^{(2,\sigma)}$, which is followed by $\mathcal{S}_{\langle \bm s_1\rangle,\langle \bm s_2\rangle}^{(2,\sigma)}={\rm span}\{\bm g,\bm s_1,\bm s_2,\bm B\bm s_1,\bm B\bm s_2\}$.

\noindent\textbf{4.} In the case where $\mathcal{W}={\rm span}\{\bm s_1,\bm s_2\}$, by the same argument in {\color{black}\textbf{3}}, we have $\mathcal{S}_{\langle \bm s_1,\bm s_2\rangle}^{(2,0)}={\rm span}\{\bm g,\bm B\bm s_1,\bm B\bm s_2\}$ and $\mathcal{S}_{\langle \bm s_1,\bm s_2\rangle}^{(2,\sigma)}={\rm span}\{\bm g,\bm s_1,\bm s_2,\bm B\bm s_1,\bm B\bm s_2\}$.
\end{proof}

By Table \ref{tab:subspace-1} and \ref{tab:subspace-2}, model-driven subspaces diverge from common linear subspaces (classic subspaces) to subsets of linear subspaces (truncated subspaces). For both types of subspaces, dimension is an important concept since it can distinguish the space complexity of different subspaces and determine the complexity of solving a trust-region subproblem inside.

Therefore, for (model-driven) truncated subspaces, we still want to talk about their dimension, and for this reason, we give the Definition \ref{def:dimension of subspace}. We write $\dim(\mathcal{L})$ denote the dimension of $\mathcal{L}$ when $\mathcal{L}$ is a linear subspace of $\mathbb{R}^n$.

\begin{definition}\label{def:dimension of subspace}
    For truncated subspace $\mathcal{S}$, let $\mathcal{M}=\{\bm x+\mathcal{L}:\mathcal{S}\subseteq \bm x+\mathcal{L} \}$ be the collection of linear manifolds that includes $\mathcal{S}$. Then the dimension of $\mathcal{S}$, denoted by $\dim(\mathcal{S})$, is defined by 
    \begin{equation*}
        \dim(\mathcal{S})=\min_{\bm x+\mathcal{L}\in\mathcal{M}}\dim(\mathcal{L}).
    \end{equation*}
\end{definition}

A simple observation shows that when the truncated subspace $\mathcal{S}$ is exactly a linear subspace, its dimension is exactly its dimension as a linear subspace. In this sense, we believe this term is well defined. By Table \ref{tab:subspace-2} and the proof of Theorem \ref{thm:sp. gene. by span+cup}, dimensions of the second type of subspaces are obvious since they are all linear subspaces of $\mathbb{R}^n$. We carefully examine the dimensionality of $\mathcal{S}_{\mathcal{B}}^{(1,\sigma)}$ in the following Proposition \ref{prop: dim of sp1}.
\begin{proposition}\label{prop: dim of sp1}
For $\sigma\geq 0$, we have
    \begin{equation*}  \dim(\mathcal{S}_{\mathcal{B}}^{(1,\sigma)})=n.
    \end{equation*}
\end{proposition}

\begin{proof}{Proof}
    According to the proof of Theorem \ref{thm: sp. gene. by cup+span}.1, when $\sigma=0$, $\nabla p_k$ is a homeomorphism when $\sigma=0$, hence there must be $\delta>0$ such that $\mathcal{B}(\bm g;\delta)\subseteq \nabla p_k(\mathcal{B})$ (therefore $\mathcal{B}(\bm g;\delta)\subseteq\mathcal{S}_{\mathcal{B}}^{(1,0)}$), which implies that $\mathcal{S}_{\mathcal{B}}^{(1,0)}$ contains $n$ vectors $\bm g+\bm s_1,\cdots,\bm g+\bm s_n\in \mathcal{B}(\bm g;\delta)$ such that $\bm s_1,\cdots,\bm s_n$ are linearly independent. Let $\bm x+\mathcal{L}$ be an arbitrary linear manifold that contains $\mathcal{S}_{\mathcal{B}}^{(1,0)}$. Then $\mathcal{S}_{\mathcal{B}}^{(1,0)}-\bm x$ is contained in the linear subspace $\mathcal{L}$. Moreover, by $\bm g\in\mathcal{B}(\bm g;\delta)$, we have $\bm g-\bm x,\ \bm g+\bm s_i-\bm x\in\mathcal{L},\ \forall i\in[n]$. By the linearity of $\mathcal{L}$, $\bm s_1,\cdots,\bm s_n\in\mathcal{L}$, which implies that $\dim(\mathcal{L})=n$. The result is followed by Definition \ref{def:dimension of subspace}. 
    
    When $\sigma>0$, $\nabla p_k$ is homeomorphism when restricted to the sphere $\mathcal{S}_t\subseteq \mathcal{B}(\bm 0;1)$ where $t\in(0,1)$. Extend the restricted function $\nabla p_k:\mathcal{S}_t\to\nabla p_k(\mathcal{S}_t), s\mapsto \bm g+(\bm B+\sigma t\bm I)\bm s$ to function $h: \overline{\mathcal{B}(\bm 0;t)}\to h(\overline{\mathcal{B}(\bm 0;t)}), \bm s\mapsto \bm g+(\bm B+\sigma t\bm I)\bm s$. Then it is easy to verify that $h$ is still a homeomorphism. By the argument in the case $\sigma=0$, we can find $n$ linearly independent vectors $\bm y_1,\cdots,\bm y_n$ such that $\bm g+\bm y_i\in h(\overline{\mathcal{B}(\bm 0;t)})$. Let $\bm y_i'=t\bm y_i/\Vert (\bm B+\sigma t\bm I)^{-1}\bm y_i\Vert$, then $\bm y_i'$ is a proper scaling of $\bm y_i$ such that $\bm y_i'\in\mathcal{S}_t$ and $h(\bm y_i')=\nabla p_k(\bm y_i')\in\mathcal{S}_{\mathcal{B}}^{(1,\sigma)}$. Since $\bm y_i'$ are linearly independent, let $\bm s_i'=(\bm B+\sigma t\bm I)\bm y_i'$, then $\bm s_i'$ are also linearly independent while $\bm g,\ \bm g+\bm s_i'\in\mathcal{S}_{\mathcal{B}}^{(1,\sigma)}$. Apply the same argument as in the case $\sigma=0$, the result follows.
\end{proof}

 We give the results of other first-type subspaces in Table \ref{tab:dim of subspaces} without proof details, which is straightforward.

\begin{table}[htbp]
    \centering
    \caption{Dimension of the first type of subspaces}
    \centering
    \begin{tabular}{m{1cm}<{\centering}m{2cm}
    <{\centering}m{2cm}
    <{\centering}m{2cm}<{\centering}m{2.2cm}<{\centering}}
    \toprule
     & $\dim (\mathcal{S}_{\mathcal{B}}^{(1,\sigma)})$ & $\dim (\mathcal{S}_{l}^{(1,\sigma)})$ & $\dim (\mathcal{S}_{\bm s_1,\bm s_2}^{(1,\sigma)})$ & $\dim (\mathcal{S}_{\langle\bm s_1,\bm s_2\rangle}^{(1,\sigma)})$
     \\
    \midrule
        $\sigma=0$  & $n$ & 2 & $\in\{1,2\}$ & 3  \\
        

        $\sigma>0$ & $n$ & 2 & $\in\{1,2\}$ & $\in\{1,2,3,4,5\}$ \\
        
    \bottomrule

    \end{tabular}   
    \label{tab:dim of subspaces}
\end{table}

\subsection{Subproblem on truncated subspaces}\label{section: Trust region subproblem on subspace}

In this section, we introduce our strategy for solving the subproblem in our truncated subspaces, and we especially analyze the effect of the algorithm on the two-dimensional truncated subspace
$\mathcal{S}_l^{(1,\sigma)}$. Throughout this section, by subspaces, we always refer to the projected subspace $\hat{\mathcal{S}_k}$ unless it is specially noted.

To be specific, suppose $\hat{\mathcal{S}_k}$ is an $q$-dimensional subspace, in solving the subproblem, classic methods consider finding a point $\bm y^*$ such that
\begin{equation}\label{prob: fin min of ARC}
    m_k(\bm y^*)\approx\arg\min_{\bm y\in\mathbb{R}^q}m_k(\bm y)=f(\bm x_k)+\bm{g}_k^\top\bm{y}+\frac{1}{2}\bm{y}^\top \bm{B}_k\bm{y}+\frac{\sigma_k}{3}\Vert \bm{y}\Vert^3,
\end{equation}
where $\bm g_k=\bm P_k^\top\nabla f(\bm x_k)$ and $\bm B_k=\bm P_k^\top\nabla^2f(\bm x_k)\bm P_k$. By \citet{CartisGouldToint2011}, Problem \eqref{prob: fin min of ARC} is equivalent to the following trust-region subproblem \eqref{prob: standard TRS} when $\sigma_k=c_k/\Delta_k$ for some constant $c_k$.
\begin{equation}\label{prob: standard TRS}
   \begin{aligned}
     \min_{\bm y\in\mathbb{R}^q}\ &m_k(\bm y)=f(\bm x_k)+\bm{g}_k^\top\bm{y}+\frac{1}{2}\bm{y}^\top \bm{B}_k\bm{y}\\
    \text{subject to}\ & \bm \Vert \bm y\Vert
    \leq \Delta_k. 
\end{aligned} 
\end{equation}
Therefore, to simplify the discussion, we adopt the above trust-region description of the subproblem and refer $m_k(\bm y)$ to the quadratic model $f(\bm x_k)+\bm{g}_k^\top\bm{y}+\frac{1}{2}\bm{y}^\top \bm{B}_k\bm{y}$ for the rest of this section. For the rest of this paper, we will call such $m_k$ as ``quadratic $m_k$" while the one with $\sigma_k\Vert \bm y\Vert^3/3$ term is called ``cubic $m_k$".

To solve the unconstrained version of problem \eqref{prob: standard TRS}, classic conjugate gradient method generates a sequence $\{\bm y_{k,i}\}_{i\geq0}$ at the $k$th iteration, where
\begin{equation}\label{eq: y_k and d_k}
    \begin{aligned}
    \bm{y}_{k,i+1}&=\bm{y}_{k,i}+\alpha_i \bm{d}_i,\\
    \bm{d}_{i+1}&=-\bm{g}_{k,i+1}+\beta_i \bm{d}_i,
    \end{aligned}
\end{equation}
\begin{equation}\label{eq: g_k}
    \bm{g}_{k,i+1}=\nabla \hat{m_k}(\bm{y}_{k,i+1})=\bm{B}_k\bm{y}_{k,i+1}+\bm{g}_k
\end{equation} 
 and 
\begin{equation}\label{eq: alpha and beta}
    \alpha_i=-\frac{\bm{g}_{k,i}^\top\bm{d_i}}{\bm{d}^\top_i\bm{B}_k\bm{d}_i},\ \beta_i=\frac{\Vert \bm{g}_{k,i+1}\Vert^2}{\Vert \bm{g}_{k,i}\Vert^2}. 
\end{equation}
At each inner iteration, this method searches for the next point along the direction that is conjugate to the last one. Providing the problem is quadratic, this method guarantees termination within $q$ steps when solving the unconstrained problem, provided the problem is $q$-dimensional. For the 2-dimensional subspace (meaning the problem is also 2-dimensional), especially, we calculate the first two inner iterates and find they lie exactly within our 2-dimensional truncated subspace $\hat{\mathcal{S}}_{l}^{(1,\sigma)}$ in some cases. We state and prove this result in Lemma \ref{lem: y_1y_2 in S}.

\begin{lemma}\label{lem: y_1y_2 in S}
Let $\bm y_{k,0}=\bm 0$ ,$\bm d_0=-\bm g_k$. If $\nabla^2 f(\bm x_k)$ is positive definite and 
\[
\frac{\Vert \bm{g}_k\Vert^2}{\bm{g}_k^\top\bm{B}_k\bm{g}_k}\leq\frac{1}{2\sqrt{\sigma\Vert\bm{g}_k\Vert}},
\] 
then $\bm y_{k,1},\bm y_{k,2}\in \hat{\mathcal{S}}_l^{(1,\sigma)}$. 
\end{lemma}

\begin{figure}[H]
    \centering
    \includegraphics[width=0.8\linewidth]{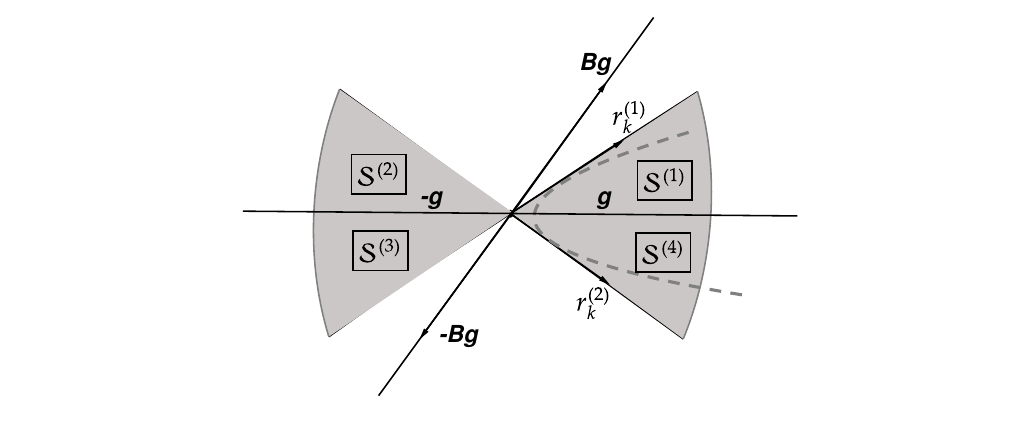}
    \caption{Partition of subspace $\hat{\mathcal{S}}_{l}^{(1,\sigma)}$ 
    }
    \label{fig:parts of subspace S}
\end{figure}

\begin{proof}{Proof}
Since $\hat{\mathcal{S}}_l^{(1,\sigma)}\subseteq\hat{\mathcal{S}}_l^{(1,0)}$, it suffices to prove the result for $\sigma>0$.

First of all, since ${\rm span} \{\bm{g}_k\}\subseteq \hat{\mathcal{S}}_l^{(1,\sigma)}$ and
    \[
    \bm{y}_{k,1}=\bm{y}_{k,0}+\alpha_0 \bm{d}_0=\frac{\Vert \bm{g}_k\Vert^2}{\bm{g}_k^\top\bm{B}_k\bm{g}_k}(-\bm{g}_k), 
    \] 
    we know that $\bm y_{k,1}\in \hat{\mathcal{S}}_l^{(1,\sigma)}$. By $\bm B_k=\bm P_k^\top\nabla^2f(\bm x_k)\bm P_k$, $\nabla^2f(\bm x_k)$ is positive definite implies that $\bm B_k$ is also positive definite.

Establish the affine coordinated system using $\{\bm O,\bm g_k,\bm B_k\bm g_k\}$ as the frame. Let $\bm g_k$ be the $x$-axis and $\bm B_k\bm g_k$ be the $y$-axis and $\bm O$ be the origin.
See Figure \ref{fig:parts of subspace S}. By the proof of Theorem \ref{thm: sp. gene. by cup+span}, $\hat{\mathcal{S}}_l^{(1,\sigma)}$ are bounded by lines (of directions) $\hat{r}_k^{(1)}$ and $\hat{r}_k^{(2)}$ where 
$$\hat{r}_k^{(1)}:\ y=\frac{1}{2\sqrt{\sigma \Vert\bm g_k\Vert}}x,\text{ and }\hat{r}_k^{(2)}:\  y=-\frac{1}{2\sqrt{\sigma \Vert\bm g_k\Vert}}x.$$
Let $\mathcal{S}_k^{(1)}$ be the subregion of $\hat{\mathcal{S}}_l^{(1,\sigma)}$ which includes all positive linear combinations of $\bm g_k$ and $\hat{r}_k^{(1)}$, $\mathcal{S}_k^{(2)}$ be the subregion which consists of all positive linear combinations of $-\bm g_k$ and $-\hat{r}_k^{(2)}$, $\mathcal{S}_k^{(3)}$ be the subregion of consists of all positive linear combinations of $-\bm g_k$  and $-\hat{r}_k^{(1)}$, and finally, let $\mathcal{S}_k^{(4)}$ be the subregion of $\hat{\mathcal{S}}_l^{(1,\sigma)}$ which consists of all positive linear combinations of $\bm g_k$ and $\hat{r}_k^{(2)}$.

Through this way, we have divided $\hat{\mathcal{S}}_l^{(1,\sigma)}$ into $\hat{\mathcal{S}}_l^{(1,\sigma)}=\cup_{i=1}^{4}\mathcal{S}_k^{(i)}$ where in the affine coordinate system
\begin{equation*}
    \begin{aligned}
        \mathcal{S}_k^{(1)}&=\{(x,y):x\geq0,0\leq y\leq \frac{1}{2\sqrt{\sigma \Vert\bm g_k\Vert}} x\},\
        \mathcal{S}_k^{(2)}=\{(x,y):x\leq 0,0\leq y\leq -\frac{1}{2\sqrt{\sigma \Vert\bm g_k\Vert}} x\},\\
        \mathcal{S}_k^{(3)}&=\{(x,y):x\leq0,\frac{1}{2\sqrt{\sigma \Vert\bm g_k\Vert}} x\leq y\leq 0\},\
        \mathcal{S}_k^{(4)}=\{(x,y):x\geq0,-\frac{1}{2\sqrt{\sigma \Vert\bm g_k\Vert}} x\leq y\leq 0\}.\\
    \end{aligned}
\end{equation*}

Then the linear combination of any two points that lie in the same $\mathcal{S}_k^{(i)}$ still belongs to this part. In the following paragraphs, we will use this property to deduce $\bm y_{k,2}\in\mathcal{S}_k^{(2)}$ from $\bm y_{k,1},\bm d_1\in\mathcal{S}_k^{(2)}$, which completes the proof.

    According to the definition of $\bm g_{k,i}$, $\bm d_i$, $\alpha_i$ and $\beta_i$,   
    \begin{equation}\label{eq: g1,beta0,d1 in lemma}
    \begin{aligned}
        \bm{g}_{k,1}&=\nabla m_k(\bm{y}_{k,1})=\bm{B}_k\bm{y}_1+\bm{g}_k=-\alpha_0\bm{B}_k\bm{g}_k+\bm{g}_k=(1,-\alpha_0),\\       \beta_0&=\frac{\Vert\bm{g}_{k,1}\Vert^2}{\Vert\bm{g}_k\Vert^2}, 
        \bm{d}_1=-\bm{g}_{k,1}+\beta_0\bm{d}_0=-(1+\beta_0)\bm{g}_k+\alpha_0\bm{B}_k\bm{g}_k=(-(1+\beta_0),\alpha_0).
    \end{aligned}
    \end{equation}
    Since
    \[
    \alpha_0=\frac{\Vert \bm{g}_k\Vert^2}{\bm{g}_k^\top\bm{B}_k\bm{g}_k}\leq\frac{1}{2\sqrt{\sigma\Vert\bm{g}_k\Vert}},
    \]
    we have $\bm g_{k,1}=(1,-\alpha_0)\in\mathcal{S}_k^{(4)}$. By \eqref{eq: g1,beta0,d1 in lemma}, $\beta_0\geq 0$. Therefore
    \begin{equation*}
        \alpha_0\leq\frac{1}{2\sqrt{\sigma \Vert\bm g_k\Vert}}\leq \frac{(1+\beta_0)}{2\sqrt{\sigma \Vert\bm g_k\Vert}},
    \end{equation*}
which implies $\bm d_1=\left(-(1+\beta_0),\alpha_0\right)\in\mathcal{S}_k^{(2)}$.

    Finally, since 
    $\alpha_1=-\bm{g}_{k,1}^\top \bm{d}_1/(\bm{d}_1\bm{B}_k\bm{d}_1)$, where 
    \begin{equation*}
        -\bm{g}_{k,1}^\top \bm{d}_1=-\bm{g}_{k,1}^\top(-\bm{g}_{k,1}+\beta_0\bm{d}_0)=\Vert \bm{g}_{k,1}\Vert^2+0\geq 0, 
    \end{equation*}
    we have $\alpha_1\geq 0$, therefore $\alpha_1\bm d_1\in\mathcal{S}_k^{(2)}$. Together with the fact that $\bm y_{k,1}=-\alpha_0\bm g_k=(-\alpha_0,0)\in\mathcal{S}_k^{(2)}$, we can deduce that $\bm y_{k,2}=\bm y_{k,1}+\alpha_1\bm d_1\in\mathcal{S}_k^{(2)}\subseteq \hat{\mathcal{S}}_l^{(1,\sigma)}$.    
        
\end{proof}

Under the conditions of Lemma \ref{lem: y_1y_2 in S}, in order to find the minimizer of $\hat{m_k}$, which belongs to $\{\bm y_{k,1},\bm y_{k,2}\}$, it suffices to search the truncated subspace $\mathcal{S}_l^{(1,\sigma)}$ instead of the complete linear subspace $\mathbb{R}^2$. This inspires and justifies us to find a solution within the truncated subspace under the truncation settings. This means that we are actually considering solving the following ``reduced" trust-region subproblem \eqref{prob:TRS on S} of quadratic model $m_k$ on the truncated $q$-dimensional subspace $\hat{\mathcal{S}}_k$.
\begin{equation}\label{prob:TRS on S}
\begin{aligned}
    \min_{\bm y\in\mathbb{R}^q}\ &m_k(\bm y)=f(\bm x_k)+\bm{g}_k^\top\bm{y}+\frac{1}{2}\bm{y}^\top \bm{B}_k\bm{y}\\
    \text{subject to}\ & \Vert \bm y\Vert\leq \Delta_k\text{ and }\bm y\in\hat{\mathcal{S}}_k. 
\end{aligned}    
\end{equation}
Based on the classic truncated conjugate gradient method (see Algorithm 1 in \citet{2000On}), we designed Algorithm \ref{algo:tcg with radius, steps form} to deal with extra boundary conditions and to solve problem \eqref{prob:TRS on S}.
\begin{algorithm}
\caption{STCG}
\begin{algorithmic}[1]
    \State \textbf{Step 0:} Input subspace $\hat{\mathcal{S}}_k$, $\bm y_{k,0}=\bm{0}, \bm g_{k,0}=\bm{g}_k, \bm d_0=-\bm{g}_k\in\mathbb{R}^q$ where $q=\dim(\hat{\mathcal{S}}_k)$, trust-region radius $\Delta_k$, $i=0$.

    \State \textbf{Step 1:} If $\Vert \bm g_{k,i}\Vert=0$, output $\bm y_{k,i}$ and stop. Compute $\bm{d}_i^\top\bm{B}_k\bm{d}_i$, if $\bm{d}_i^\top\bm{B}_k\bm{d}_i\leq 0$, go to \textbf{Step 4}; Compute $\alpha_i$ according to \ref{eq: alpha and beta}.

    \State \textbf{Step 2:} If $\bm y_{k,i}+\alpha_i\bm d_i\notin\hat{\mathcal{S}}_k$, go to \textbf{Step 5}.

    \State \textbf{Step 3:} If $\Vert \bm y_{k,i}+\alpha_i\bm d_i\Vert>\Delta_k$, go to \textbf{Step 4}; Set $\bm y_{k,i+1}$ and $\bm g_{k,i+1}$ by \ref{eq: y_k and d_k} and \ref{eq: g_k}. Compute $\beta_i$ by \ref{eq: alpha and beta} and set $\bm d_{i+1}$ according to \eqref{eq: y_k and d_k}.
    Increment $i$ by 1 and go to \textbf{Step 1}.
    \State \textbf{Step 4:} Compute $\alpha^*\geq 0$ such that $\Vert \bm{y}_{k,i}+\alpha^*\bm{d}_i\Vert=\Delta_k$. Return $\bm{y}_{k,i}+\alpha^*\bm{d}_i$ and stop.

    \State \textbf{Step 5:} Compute $\alpha^*\geq 0$ such that $\bm{y}_{k,i}+\alpha^*\bm{d}_i\in\partial \hat{\mathcal{S}}$\footnote{In numerical practice, $\bm{y}_{k,i}+\alpha^*\bm{d}_i\in\hat{\mathcal{S}}$ is selected to be the nearest point from $\bm{y}_{k,i}+\alpha_i\bm{d}_i$ to projected subspace $\hat{\mathcal{S}}$}. Return $\bm{y}_{k,i}+\alpha^*\bm{d}_i$ and stop.
\end{algorithmic}
\label{algo:tcg with radius, steps form}
\end{algorithm}
In fact, Algorithm \ref{algo:tcg with radius, steps form} only differs from the classic truncated conjugate gradient algorithm at Step 2 and 5 for the truncation of the input subspace $\mathcal{S}$. As a consequence, Algorithm \ref{algo:tcg with radius, steps form} will always terminate no later than the classic truncated conjugate gradient method, which we believe can save CPU time greatly. Moreover, we found that the solution obtained by Algorithm \ref{algo:tcg with radius, steps form} also achieves sufficient model-function-value decrease, where the sufficient decrease of the classic truncated conjugate gradient method is proved by \citet{2000On}.
\begin{theorem}[\citet{2000On}]
    Consider problem \eqref{prob: standard TRS}, for any $\Delta_k>0$. $\bm g_k\in\mathbb{R}^q$ and any positive definite matrix $\bm B_k\in\mathbb{R}^{q\times q}$, let $\bm s^*$ be the global solution of the trust-region subproblem \eqref{prob: standard TRS}, let $\bm y^*$ be the solution obtained by truncated conjugate gradient method (Algorithm 1 in \cite{2000On}), then
    \begin{equation*}
        m_k(\bm y^*)\leq \frac{1}{2}m_k(\bm s^*).
    \end{equation*}
\end{theorem}
For two-dimensional problems, the positive-definite condition can be relaxed \cite{2000On}.

\begin{theorem}[\citet{2000On}]
    If $q=2$, let $\bm y^*$ be the solution of problem \eqref{prob: standard TRS} obtained by the truncated conjugate gradient method, let $\bm s^*$ be the global solution of problem \eqref{prob: standard TRS}. Then for any $\Delta_k>0$, $\bm g_k\in\mathbb{R}^q$ and any symmetric matrix $\bm B_k\in\RR^{q\times q}$, we have 
    \begin{equation*}
    m_k(\bm 0)-m_k(\bm y^*)\geq \frac{1}{2}\left(m_k(\bm 0)-m_k(\bm s^*)\right).
\end{equation*}
\end{theorem}

Therefore, by our previous analysis, when the problem is of dimension two, the output of solving the problem \eqref{prob:TRS on S} using Algorithm \ref{algo:tcg with radius, steps form} is exactly the same as the result of solving problem \eqref{prob: standard TRS} using the standard truncated conjugate gradient method. It follows that the solution obtained by our modified algorithm STCG also preserves a certain fraction of decrease as stated in Theorem \ref{thm:algo3 achieved certain decrease}.
\begin{theorem}\label{thm:algo3 achieved certain decrease}
When $\mathcal{S}=\mathcal{S}_{l}^{(1,\sigma)}$, for any $\Delta_k>0$, $\bm g_k\in\mathbb{R}^2$ and symmetric matrix $\bm B_k\in\mathbb{R}^{2\times 2}$. Let $\bm y^*$ be the solution of problem \eqref{prob:TRS on S} obtained by Algorithm \ref{algo:tcg with radius, steps form}, let $\bm s^*$ be the global solution of \eqref{prob:TRS on S} 
, then 
\begin{equation}\label{ineq:1/2 decrease of our algo}
    m_k(\bm 0)-m_k(\bm y^*)\geq \frac{1}{2}\left(m_k(\bm 0)-m_k(\bm s^*)\right).
\end{equation}
\end{theorem}


\begin{proof}{Proof}
    It is clear that \textbf{Step 1, 3, 4} consists of the standard TCG algorithm. By Lemma \ref{lem: y_1y_2 in S}, $\bm y_{k,1},\bm y_{k,2}\in\hat{\mathcal{S}}$. Hence, if the problem is two-dimensional, the iteration sequence will never go to \textbf{Step 2} and \textbf{Step 5}, and the output sequence will coincide with the sequence obtained by the standard TCG method in solving problem \eqref{prob: standard TRS}.
    Let $\bm x^*$ be the exact solution of problem \eqref{prob: standard TRS}, then 
    \begin{equation}\label{ineq: inex_sol of our prob/ex_sol of TRS}
       m_k(\bm 0)-m_k(\bm y^*)\geq \frac{1}{2}\left(m_k(\bm 0)-m_k(\bm x^*)\right).
    \end{equation} 
    Since $\hat{\mathcal{S}}_k\cap\mathcal{B}(\bm x_k;\Delta)\subseteq \mathcal{B}(\bm x_k;\Delta)$,  \begin{equation}\label{ineq:ex_sol of TRS/ex_sol of our prob}
        m_k(\bm 0)-m_k(\bm x^*)\geq m_k(\bm 0)-m_k(\bm s^*).
    \end{equation}
    Combining \eqref{ineq: inex_sol of our prob/ex_sol of TRS} and \eqref{ineq:ex_sol of TRS/ex_sol of our prob} gives\begin{equation*}
    m_k(\bm 0)-m_k(\bm y^*)\geq \frac{1}{2}\left(m_k(\bm 0)-m_k(\bm s^*)\right).  \end{equation*}
\end{proof}

Besides the two-dimensional truncated subspace $\hat{\mathcal{S}}_{l}^{(1,\sigma)}$, it can be easily seen that the solution of the trust-region subproblem on all type-2 truncated subspace possesses a 1/2 decrease. With an additional assumption (Assumption \ref{assumption: cauchy}), we conclude in Corollary \ref{cor: cauchy decrease, cor of thm 9} that solutions on such truncated subspaces satisfy the Cauchy decrease condition.
\begin{assumption}\label{assumption: cauchy}
    The Hessian $\nabla^2m_k(\bm 0)$ of model $m_k(\bm s)$
    is positive definite. Denote $q=\dim(\mathcal{S}_k)$.
    Let $\lambda_1,\cdots,\lambda_q$ be the eigenvalues of $\nabla^2m_k(\bm 0)$, then 
    \begin{equation*}  \left(\sum_{i=1}^q\lambda_i^2\right)\left(\sum_{i=1}^q\frac{1}{\lambda_i}\right)\leq c,
    \end{equation*}
    where $c\in[1/4,1/2]$ is a constant independent of $k$.
\end{assumption}

In fact, with Assumption \ref{assumption: cauchy}, we can prove a model value decrease which is slightly stronger than the Cauchy-type decrease.
\begin{theorem}
Suppose $m_k$ satisfies Assumption \ref{assumption: cauchy} and \eqref{ineq:1/2 decrease of our algo}, then the computed solution $\bm y^*$ satisfies
\begin{equation}\label{ineq:pre_cauchy decrease}
    m_k(\bm 0)-m_k(\bm y^*)\geq c_1\Vert\nabla m_k(\bm 0)\Vert^2
\end{equation}
for some constant $c_1\in[1/2,1]$ independent of $k$.
\end{theorem}

\begin{proof}{Proof}
    Since $m_k$ satisfies \eqref{ineq:1/2 decrease of our algo}, \begin{equation*}
    \begin{aligned}
        m_k(\bm 0)-m_k(\bm y^*)
        &\geq\frac{1}{2}\left((m_k(\bm 0)-m_k(\bm s^*)\right)\\
        &=-\frac{1}{2}\left(\bm g^\top \bm s^*+\frac{1}{2}\bm s^{*\top}\bm B\bm s^*\right)\\
        &=-\frac{1}{2}\left(\nabla m_k(\bm s^*)-\frac{1}{2}\bm B\bm s^*\right)^\top\bm s^*,
    \end{aligned}
    \end{equation*}
where $\bm s^*$ is the global minimizer of $m_k$. Therefore $\nabla m_k(\bm s^*)=\bm 0$ and 
\begin{equation*}
        -\frac{1}{2}\left(\nabla m_k(\bm s^*)-\frac{1}{2}\bm B_k\bm s^*\right)^\top\bm s^*=\frac{1}{4}\bm s^{*\top}\bm B_k\bm s^*.
\end{equation*}
By our assumption, $\bm B_k=\nabla^2m_k(\bm 0)=\nabla^2m_k(\bm s^*)$ is positive definite. Hence there is a orthogonal matrix $\bm Q\in\mathbb{R}^{q\times q}$ such that
\begin{equation*}
    \bm Q^\top \bm B_k\bm Q={\rm diag}\{\lambda_1,\cdots,\lambda_q\},\ \lambda_i>0.
\end{equation*}

Let $\bm s^*=\bm Q \bm s$, $\bm g_k=\bm Q\bm y$, where $\bm s=(s_1,\cdots,s_q)^\top$ and $\bm y=(y_1,\cdots,y_q)^\top$, then 
\begin{equation*}
\begin{aligned}
     \frac{1}{4}\bm s^{*\top}\bm B_k\bm s^*=\frac{1}{4}\bm s^\top\bm Q^\top\bm B_k\bm Q\bm s=\frac{1}{4}\sum_{i=1}^{q}\lambda_i s_i^2.
\end{aligned}
\end{equation*}
By Cauchy's inequality, we have 
\begin{equation*}
    \sum_{i=1}^{q}\lambda_i s_i^2\geq \frac{\left(\sum_i^q| s_i|\right)^2}{\sum_{i=1}^q\frac{1}{\lambda_i}}\geq\frac{\Vert\bm s\Vert^2}{\sum_{i=1}^q\frac{1}{\lambda_i}}.
\end{equation*}
Hence 
\begin{equation*}
\begin{aligned}
     \frac{1}{4}\sum_{i=1}^{q}\lambda_i s_i^2\geq\frac{1}{4}\frac{\Vert\bm s\Vert^2}{\sum_{i=1}^q\frac{1}{\lambda_i}}=\frac{1}{4\sum_{i=1}^q\frac{1}{\lambda_i}}\Vert\bm Q^\top \bm s^*\Vert.
\end{aligned}
\end{equation*} 
Moreover, since $\Vert \bm Q^\top\bm s^*\Vert=\Vert \bm Q^\top\bm B_k^{-1}(-\bm g_k)\Vert=\sum_{i=1}^q y_i^2/\lambda_i^2$, we have
\begin{equation*}
    \frac{1}{4}\sum_{i=1}^{q}\lambda_i s_i^2\geq\frac{1}{4\sum_{i=1}^q\frac{1}{\lambda_i}}\Vert\bm Q^\top \bm B_k^{-1}\left(-\bm g_k\right)\Vert=\frac{1}{4\sum_{i=1}^q\frac{1}{\lambda_i}}\sum_{i=1}^q \left(\frac{y_i}{\lambda_i}\right)^2.
\end{equation*}

Again, applying Cauchy's inequality, we have
\begin{equation*}
    \sum_{i=1}^q \left(\frac{y_i}{\lambda_i}\right)^2\geq \frac{(\sum_i^q|y_i|)^2}{\sum_{i=1}^q\lambda_i^2}\geq \frac{\Vert\bm y\Vert^2}{\sum_{i=1}^q\lambda_i^2}.
\end{equation*} 
Therefore
\begin{equation*}
\begin{aligned}
    \frac{1}{4\sum_{i=1}^q\frac{1}{\lambda_i}}\sum_{i=1}^q \left(\frac{y_i}{\lambda_i}\right)^2\geq \frac{1}{4\sum_{i=1}^q\frac{1}{\lambda_i}}\frac{\Vert\bm y\Vert^2}{\sum_{i=1}^q\lambda_i^2}\geq \frac{1}{4\sum_{i=1}^q\frac{1}{\lambda_i}}\frac{\Vert\bm g_k\Vert^2}{\sum_{i=1}^q\lambda_i^2},
\end{aligned}  
\end{equation*}
where the second inequality comes from 
\begin{equation*}
    \Vert \bm g_k\Vert=\Vert\bm Q\bm y\Vert\leq \Vert\bm Q\Vert\Vert\bm y\Vert=\Vert\bm y\Vert,
\end{equation*}
$\Vert\bm Q\Vert=\sup_{\Vert\bm x\Vert=1}\Vert\bm Q\bm x\Vert$ is the matrix norm of $\bm Q$ and $\Vert\bm Q\Vert=1$ since $\bm Q$ is an orthogonal matrix. 

Let $c_1=\frac{1}{4c}$, then by Assumption \ref{assumption: cauchy}, $c_1\in[1/2,1]$ and 
\begin{equation*}
     c_1\leq\frac{1}{4}\frac{1}{\sum_{i=1}^m\frac{1}{\lambda_i}\sum_{i=1}^m\lambda_i^2}.
\end{equation*}
The proof is completed by observing $\bm g_k=\nabla m_k(\bm 0)$.  
\end{proof}

\begin{corollary}\label{cor: cauchy decrease, cor of thm 9}
    If $m_k$ satisfies Assumption \ref{assumption: cauchy} and \eqref{ineq:1/2 decrease of our algo}, then computed solution $\bm y^*$ satisfies the Cauchy decrease condition
    \begin{equation*}
        m_k(\bm 0)-m_k(\bm y^*)\geq c_1\Vert\bm \nabla m_k(\bm 0)\Vert\min\left(\Delta_k,\frac{\Vert\bm \nabla m_k(\bm 0)\Vert}{\max(1,\Vert \nabla^2m_k(\bm 0)\Vert)}\right).
    \end{equation*}
\end{corollary}

\begin{proof}{Proof}
    The result is easily followed by \eqref{ineq:pre_cauchy decrease}.  
\end{proof}

We will show in Section \ref{section: convergence} that with the Cauchy decrease in Corollary \ref{cor: cauchy decrease, cor of thm 9} at every successful step and some constraints on $f$, the sequence generated by MD-LAMBO will finally converge to a stationary point $\bm x^*$ such that $\Vert \nabla f(\bm x^*)\Vert=0$.

\subsection{Low-dimensional subspace models for derivative-free algorithm}\label{section: Model functions for derivative-free algorithm}

In this section, we provide one class of functions as the alternative for model $p_k$ in a derivative-free version of Algorithm \ref{algo: complete algorithm}. Since we want our model to approximate the object function in the aspects of function value, gradient, and {\color{black}Hessian}, we adopt the $\bm P$-fully quadratic model introduced by \citet{Cartis2023}. The following definition is introduced in \cite{1chen2024q}
\begin{definition}\label{P-fully_quadratic}
    Assume that $f: \RR^n \rightarrow \RR$ is a continuous differentiable objective function whose gradient is Lipschitz continuous, $\bm{P} \in \RR^{n \times q}$. Given $\Delta > 0$, model $m_{\Delta}$ is called $\bm P$-fully quadratic (with respect to $f$) in $\mathcal{B}(\bm x;\Delta)\subset\mathbb{R}^n$ if there exists $\kappa_{ef}, \kappa_{eg}, \kappa_{eh} > 0$, such that 

    (1) the error between the Hessian of the model and the Hessian of the function satisfies
    $$
    \|\ml{P}^{\top} \nabla^2 f(\ul{x + P y}) \ml{P} -\nabla^2 m_{\Delta}(\ul{y})\| \leq \kappa_{eh} \Delta, \quad \forall \,   \ul{y} \in \mathcal{B}(\ul{0} ; \Delta),
    $$

    (2) the error between the gradient of the model and the gradient of the function satisfies
    $$
    \|\ml{P}^{\top} \nabla f(\ul{x + P y})-\nabla m_{\Delta}(\ul{y})\| \leq \kappa_{e g} \Delta^2, \quad \forall \,   \ul{y} \in \mathcal{B}(\ul{0} ; \Delta),
    $$

    (3) the error between the model and the function satisfies  
    $$
    |f(\ul{x + Py})-m_{\Delta}(\ul{y})| \leq \kappa_{e f} \Delta^3, \quad \forall \,   \ul{y} \in \mathcal{B}(\ul{0} ; \Delta). 
    $$
When $\bm P$ is the identity matrix $\bm I$, we call the model fully quadratic as a shorthand of $\bm I$-fully quadratic. 
\end{definition}

In fact, to obtain a class of fully quadratic models in the case where the model Hessian and model gradients already satisfy approximation conditions, it suffices to let the function values be close enough at one point.

\begin{theorem}\label{theorempre}
    Assume that $f: \RR^n \rightarrow \RR$ is a continuous differentiable objective function whose Hessian is Lipschitz continuous, $\bm P \in \RR^{n \times q}$. Given $\Delta > 0$, model $m_{\Delta}$ is $\ml{P}$-fully quadratic (with respect to $f$) in $\mathcal{B}(\bm x;\Delta)\subset \RR^n$ if and only if there exists $\kappa_{ef}', \kappa_{eg}, \kappa_{eh} > 0$, such that 

    (1) The error between the Hessian of the model and the Hessian of the function satisfies
    $$
    \|\ml{P}^{\top} \nabla^2 f(\ul{x + P y}) \ml{P} -\nabla^2 m_{\Delta}(\ul{y})\| \leq \kappa_{eh} \Delta, \quad \forall \,   \ul{y} \in \mathcal{B}(\ul{0} ; \Delta),
    $$

    (2) The error between the gradient of the model and the gradient of the function satisfies
    $$
    \|\ml{P}^{\top} \nabla f(\ul{x + P y})-\nabla m_{\Delta}(\ul{y})\| \leq \kappa_{e g} \Delta^2, \quad \forall \,   \ul{y} \in \mathcal{B}(\ul{0} ; \Delta).
    $$

    (3) There is at least one point which are closed in function values,
    $$
        \exists\,  \ul{y_0} \in \mathcal{B}(\ul{0}; \Delta), \text{ such that } \norm{f(\ul{x} + \ml{P}\ul{y_0}) - m_{\Delta} (y_0)} \leq \kappa_{ef}'\Delta^3.
    $$
\end{theorem}

\begin{proof}{Proof}
    We use $\kappa_{ef} = 2\kappa_{eg}+\kappa_{ef}'$ and we prove that the condition in Definition \ref{P-fully_quadratic} is satisfied. 
    Take 
    \[f(\ul{x} + \ml{P}\ul{y}) =: \hat{f}(\ul{y}).\]
    Then
    \[
        \begin{aligned}
            &|f(\ul{x} + \ml{P}\ul{y})-m_{\Delta}(\ul{y})| \\ 
            \leq& |(f(\ul{x} + \ml{P}\ul{y})-m_{\Delta}(\ul{y})) - (f(\ul{x} + \ml{P}\ul{y_0})-m_{\Delta}(\ul{y_0}))|+\vert f(\ul{x} + \ml{P}\ul{y_0})-m_{\Delta}(\ul{y_0})\vert\\
            =& |(\hat{f} - m_{\Delta})(\ul{y}) - (\hat{f} - m_{\Delta})(\ul{y_0})|+\kappa_{ef}'\Delta^3\\
            \leq& \norm{\nabla (\hat{f} - m_{\Delta})( \ul{\xi})} \norm{\bm y - \bm y_0} + \kappa_{ef}'\Delta^3,  \ (\text{since \ } \exists\,  \,  \ul{\xi} = \theta \ul{y} + (1 - \theta) \ul{y_0}, \theta \in (0, 1) )\\
            \leq& \norm{\ml{P}^{\top} \nabla f(\ul{x} + \ml{P}\ul{\xi}) - m_{\Delta}( \ul{\xi})} \norm{\bm y - \bm y_0} + \kappa_{ef}'\Delta^3,  \ (\text{since \ } \exists\,  \,  \ul{\xi} = \theta \ul{y} + (1 - \theta) \ul{y}_0, \theta \in (0, 1) )\\
            \leq& 2\kappa_{e g} \Delta^3+ \kappa_{ef}' \Delta^3\\
            =& \kappa_{e f} \Delta^3, \quad \forall \,   \ul{y} \in \mathcal{B}(\ul{0};\Delta).
        \end{aligned}
    \]  
\end{proof}
    
Inspired by Theorem \ref{theorempre}, when focusing on quadratic models, it would be interesting to observe that Definition \ref{P-fully_quadratic} could be obtained even if we only check one point for each condition.

\begin{theorem}\label{thm:3 points->quasi-fully-quadratic}
  Assume that $f: \RR^n \rightarrow \RR$ is a continuous differentiable objective function whose Hessian is Lipschitz continuous, with Lipschitz constant $L$, $\bm P \in \RR^{n \times q}$. Given $\Delta > 0$, a \textbf{quadratic} model $m_{\Delta}$ is $\ml{P}$-fully quadratic (with respect to $f$) in $\mathcal{B}(\bm x;\Delta)\subset \RR^n$ if and only if there exists $\kappa_{ef}', \kappa_{eg}', \kappa_{eh}' > 0$, such that 

    (1) The error between the Hessian of the model and the Hessian of the function satisfies
    $$
    \exists\,  \ul{y_2} \in \mathcal{B}(\ul{0}; \Delta), \text{ such that }
    \|\ml{P}^{\top} \nabla^2 f(\ul{x + P y_2}) \ml{P} -\nabla^2 m_{\Delta}(\ul{y_2})\| \leq \kappa_{eh}' \Delta{\color{black}.}
    $$

    (2) The error between the gradient of the model and the gradient of the function satisfies
    $$
      \exists\,  \ul{y_1} \in \mathcal{B}(\ul{0}; \Delta), \text{ such that }
      \|\ml{P}^{\top} \nabla f(\ul{x + P y_1})-\nabla m_{\Delta}(\ul{y_1})\| \leq \kappa_{eg}' \Delta^2.
    $$

    (3) There is at least one point which are closed in function values,
    $$
      \exists\,  \ul{y_0} \in \mathcal{B}(\ul{0}; \Delta), \text{ such that } \norm{f(\ul{x} + \ml{P}\ul{y_0}) - m_{\Delta} (\bm y_0)} \leq \kappa_{ef}' \Delta^3.
    $$
\end{theorem}

\begin{proof}{Proof}
  (1) Take any $\ul{y} \in \mathcal{B}(\ul{0}; \Delta)$, we use $\kappa_{eh}= \norm{\bm P}^3 L + \kappa_{eh}'$ and show it satisfies the condition in Definition \ref{P-fully_quadratic}.
  \[
    \begin{aligned}
      &\|\ml{P}^{\top} \nabla^2 f(\bm x + \bm P \bm y) \ml{P} -\nabla^2 m_{\Delta}(\bm y)\|\\ 
      \leq& \|(\ml{P}^{\top} \nabla^2 f(\bm x+\bm P\bm y) \bm{P} -\nabla^2 m_{\Delta}(\bm y)) - (\bm{P}^{\top} \nabla^2 f(\bm x + \bm{P} \bm y_2) \bm{P} -\nabla^2 m_{\Delta}(\bm y_2))\| \\
      &+ \|(\bm{P}^{\top} \nabla^2 f(\bm x + \bm{P} \bm y_2) \bm{P} -\nabla^2 m_{\Delta}(\bm y_2))\|\\
      \leq& \Vert\bm{P}^{\top} \nabla^2 f(\bm x + \bm{P} \bm y) \bm{P}-\bm{P}^{\top} \nabla^2 f(\bm x + \bm{P} \bm y_2) \bm{P} \Vert + \kappa_{eh}' \Delta\\
      \leq& \norm{\bm{P}}^3 L \norm{\bm{y} - \bm{y}_2} + \kappa_{eh}' \Delta\\
      \leq& (\norm{\bm{P}}^3 L + \kappa_{eh}') \Delta.
    \end{aligned}
  \]

  (2) Take any $\ul{y} \in \mathcal{B}(\ul{0}; \Delta)$, then 
  \[
    \begin{aligned}
      &\|\ml{P}^{\top} \nabla f(\ul{x + P y})-\nabla m_{\Delta}(\ul{y})\| \\
      \leq& \|\ml{P}^{\top} \nabla f(\ul{x + P y_1})-\nabla m_{\Delta}(\ul{y_1})\|\\
      +& \|\ml{P}^{\top} \nabla (f(\ul{x + P y}) - f(\ul{x + P y_1})) -\nabla (m_{\Delta}(\ul{y}) - m_{\Delta}(\ul{y_1}))\| \\
      \leq& \kappa_{eg}' \Delta^2 + 2(\Vert\bm{P}\Vert^3 L + \kappa_{eh}') \Delta^2,
    \end{aligned}
  \]
  where the last inequality is derived by the mean-value theorem for vector-valued functions.
    
  (3) It suffices to apply Theorem \ref{theorempre}.  
  \end{proof}

\begin{proposition}\label{prop: m_k fully quadratic}
    Assume that $f:\mathbb{R}^n\to\mathbb{R}$ is twice continuously differentiable, whose Hessian is Lipschitz continuous with Lipschitz constant $L$. Let $m_k$ be the quadratic model defined in Section \ref{section: Trust region subproblem on subspace}. Then $m_k$ is $\bm P_k$-fully quadratic. That is, there exists $\kappa_{ef}, \kappa_{eg}, \kappa_{eh} > 0$ such that $m_k$ satisfies Definition \ref{P-fully_quadratic}, where
    \begin{equation*}
        m_k(\bm y)= f(\bm{x}_k)+\bm{P}_k^\top \nabla f(\bm{x}_k) \bm y+\frac{1}{2}\bm y^{\top}\bm{P}_k^\top\nabla^2f(\bm{x}_k)\bm{P}_k\bm y,
    \end{equation*}
    and $\bm{P}_k=[\bm{a}_k^{(1)},\cdots,\bm{a}_k^{(q)}]\in\mathbb{R}^{n\times q}$ is the projection matrix of {\color{black}$q$}-dimensional subspace $\mathcal{S}_k$.
\end{proposition}

\begin{proof}{Proof}
     By Theorem \ref{theorempre} and Theorem \ref{thm:3 points->quasi-fully-quadratic}, taking $\bm y_0=\bm y_1=\bm y_2=\bm 0$ immediately gives the result, where $\kappa_{eh}=L,\ \kappa_{eg}=2L,\ \kappa_{ef}=4L$.  
\end{proof}

\begin{remark}\label{remark: transitivity of fully quadratic}
    Fully-quadratic possesses some transitivity. If model $m$ is $\bm P$-fully quadratic with respect to $f$, and model $m$ is fully quadratic with respect to model $\hat{m}$, then model $m$ is also $\bm P$-fully quadratic with respect to $f$.
\end{remark}

By Remark \ref{remark: transitivity of fully quadratic} and Proposition \ref{prop: m_k fully quadratic}, cubic $m_k$ is also $\bm P_k$-fully quadratic with respect to $f$. Moreover, the proof of Proposition \ref{prop: m_k fully quadratic} shows that constants $\kappa_{ef},\kappa_{eg},\kappa_{eh}$ do not rely on the radius $\Delta_k$, implying that the quadratic $m_k$ is $\bm P_k$-fully quadratic on the whole space $\mathbb{R}^q$, therefore so does its cubic counterpart. Theorem \ref{thm: interpolation model is fully quadratic} of \citet{Audet2017DerivativeFreeAB} shows that interpolation models are fully quadratic when the interpolation set is carefully chosen, therefore if quadratic $m_k$ is obtained from a proper interpolation model $Q$, then the cubic $m_k$ is naturally $\bm P_k$-fully quadratic, which provides a practical way to construct $\bm P_k$-fully quadratic model. 
\begin{theorem}\label{thm: interpolation model is fully quadratic}
    Let $\bm x\in\mathbb{R}^n$, $f$ is twice continuously differentiable and $\nabla^2 f(\bm y)$ is Lipschitz continuous on $\mathcal{B}(\bm x;\bar{\Delta})$.  Let $\mathbb{Y}=\{\bm y^{(0)},\bm y^{(1)},\cdots,\bm y^{(l)}\}$ be poised for quadratic interpolation with $\bm y^{(0)}=\bm x$, $l=(n+1)(n+2)/2-1$, and $\Delta=\overline{\rm diam}(\mathbb{Y})\leq \bar{\Delta}$. Let $Q(\bm y)$ be the quadratic interpolation function of $f$ over $\mathbb{Y}$. Then, there exists positive $\kappa_{ef},\kappa_{eg}$ and $\kappa_{eh}$ such that for all $\bm y\in\mathcal{B}(\bm x;\Delta)$,
    \begin{equation*}
    \begin{aligned}
        \Vert f(\bm y)-Q(\bm y)\Vert\leq \kappa_{ef}\Delta^{3},\ 
        \Vert \nabla f(\bm y)-\nabla Q(\bm y)\Vert\leq \kappa_{eg}\Delta^{2},\ 
        \Vert \nabla^2 f(\bm y)-\nabla^2 Q(\bm y)\Vert\leq \kappa_{eh}\Delta.
    \end{aligned}     
    \end{equation*}
 
\end{theorem}

So far, we have only given definitions and ways of constructing fully quadratic models. For the rest of this section, we will explain from different aspects that $\bm P_k$-fully quadratic models are a proper substitute for quadratic $m_k$ in DFO settings.
Such properness is mainly depicted by the function value decrease achieved by the solution of a fully quadratic model in solving a trust-region subproblem. More specifically, let $\hat{m}_k$ be a $\bm P_k$-fully quadratic model (with respect to $f$, probably obtained through interpolation) and let $m_k$ denote the quadratic $m_k$ defined in Subsection \ref{section: Trust region subproblem on subspace}. Then $\hat{m}_k$ is fully quadratic with respect to $m_k$.

Hence, it is reasonable to assume that the difference of the two models' gradients behaves like the product of a uniformly distributed random variable and a standard normal distributed random variable, i.e.,  
\begin{equation*}
        \nabla m_k(\bm s)=\nabla \hat{m}_k(\bm s)+\kappa_{eg}\Delta^2c\frac{\bm b}{\Vert \bm b\Vert},
    \end{equation*}
    where $c\sim U[0,1]$ and the coordinates of $\bm b$ are i.i.d. standard normally distributed. We claim that under above assumptions, the expectation of $m_k(\bm 0)-m_k(\hat{\bm s})$ exactly equals to the expectation of $\hat{m}_k(\bm 0)-\hat{m}_k(\hat{\bm s})$. Theorem \ref{thm:error bound of fully quadratic models} states and proves this claim.

\begin{theorem}\label{thm:error bound of fully quadratic models}
    If 
    \begin{equation}\label{eq:nabl m and nabla hat_m}
        \nabla m_k(\bm y)=\nabla \hat{m}_k(\bm y)+\kappa_{eg}\Delta^2c\frac{\bm b}{\Vert \bm b\Vert},
    \end{equation}
    where $c\sim U[0,1]$ and the coordinates of $\bm b$ are i.i.d. standard normally distributed.    
    Let $\hat{\bm y}$ be the solution of the trust-region subproblem of $\hat{m}_k$ obtained by Algorithm \ref{algo:tcg with radius, steps form}.
    Then we have
    \begin{equation*}
        \mathbb{E}_{c,\bm b}(m_k(\bm 0)-m_k(\hat{\bm y}))=\mathbb{E}_{c,\bm b}(\hat{m}_k(\bm 0)-\hat{m}_k(\hat{\bm y})).
    \end{equation*}
\end{theorem}

\begin{proof}{Proof}
Denote $\bm b/\Vert\bm b\Vert$ by $\hat{\bm b}$, then
    \begin{equation*}
    \begin{aligned}
       \hat{m}_k(\bm 0)-\hat{m}_k(\hat{\bm y})&=\mathbb{E}(\hat{m}_k(\bm 0)-\hat{m}_k(\hat{\bm y}))\\
       &=-\mathbb{E}\left(\hat{\bm y}^\top\nabla m_k(\frac{1}{2}\hat{\bm y}) \right)+\mathbb{E}(\kappa_{eg}\Delta^2c\hat{\bm b}^\top\hat{\bm y}).
    \end{aligned}    
    \end{equation*}
It is followed by the linearity of expectation that
\begin{equation*}
\begin{aligned}
     \mathbb{E}_{c,\bm b}(\kappa_{eg}\Delta^2c\hat{\bm b}^\top \hat{\bm y})=\frac{\kappa_{eg}\Delta^2}{2}\sum_{i}y_i\mathbb{E}_{b_i}(\frac{b_i}{\Vert\bm b\Vert}),
\end{aligned}   
\end{equation*}
where {\color{black}$\bm b=(b_1\cdots,b_q)^\top$ and $\hat{\bm y}=(y_1,\cdots,y_q)^\top$}. By the independence of $b_i$'s, this implies $\mathbb{E}_{c,\bm b}(\kappa_{eg}\Delta^2c\hat{\bm b}^\top\hat{\bm y})=0$. Therefore
\begin{equation*}
\begin{aligned}
     {\color{black}\mathbb{E}_{c,\bm b}}\left(\hat{m}_k(\bm 0)-\hat{m}_k(\hat{\bm y})\right)&= -\mathbb{E}_{c,\bm b}\left(\hat{\bm y}^\top\nabla m_k(\frac{1}{2}\hat{\bm y})\right)\\
     &=\mathbb{E}_{c,\bm b}\left(m_k(\bm 0)-m_k(\hat{\bm y})\right).
\end{aligned}
\end{equation*}  

\end{proof}

We also have the following theorem, which reflects the closeness between $\hat{m}_k$ and $m_k$ from a probability aspect.
\begin{theorem}
    Under the same assumptions and notations of Theorem \ref{thm:error bound of fully quadratic models}, we have
    \begin{equation*}
        \mathbb{P}\left(\hat{m}(\bm 0)-\hat{m}(\hat{\bm y})\leq m(\bm 0)-m(\hat{\bm y})\right)=\frac{1}{2}.
    \end{equation*}
\end{theorem}

\begin{proof}{Proof}
    Since $m_k$ and $\hat{m}$ are quadratic models,  \begin{equation*}
        \hat{m}(\bm 0)-\hat{m}(\hat{\bm y})=-\hat{\bm y}^\top\nabla\hat{m}(\frac{1}{2}\hat{\bm y})\,\quad
        m(\bm 0)-m(\hat{\bm y})=-\hat{\bm y}^\top\nabla m(\frac{1}{2}\hat{\bm y}).
    \end{equation*}
    Therefore by \eqref{eq:nabl m and nabla hat_m}, we have \begin{equation*}
    \begin{aligned} \mathbb{P}\left(\hat{m}(\bm 0)-\hat{m}(\hat{\bm y})\leq m(\bm 0)-m(\hat{\bm y})\right)=\mathbb{P}(\kappa_{eg}\Delta^2c\frac{\bm b^\top}{\Vert\bm b^\Vert} \hat{\bm y}\leq 0)=\mathbb{P}(\bm b^\top\hat{\bm y}\leq 0),
    \end{aligned}     
    \end{equation*}
    where $\bm b=(b_1\cdots,b_q)$ and $\hat{\bm y}=(y_1,\cdots,y_q)$. Since $b_i\sim \mathcal{N}(0,1)$, $\sum_{i=1}^{q}b_iy_i\sim\mathcal{N}(0,\Vert\hat{\bm y}\Vert^2)$, which implies
    \begin{equation*}
        \mathbb{P}(\bm b^\top\hat{\bm y}\leq 0)=\mathbb{P}(\sum_{i=1}^{q}b_iy_i\leq 0)=\frac{1}{2}.
    \end{equation*}  
    
\end{proof}

Theorem \ref{thm:error bound of fully quadratic models} suggests that, on average, the function decrease achieved by a $\bm P_k$-fully quadratic model is the same as that of the gradient-based model. Therefore, when the gradient information of the object function is unavailable, we suggest using $\bm P_k$-fully quadratic models as an alternative. However, since the definition of fully quadratic involves {\color{black}``$\forall \bm y\in \mathcal{B}(\bm 0;\Delta)$''} conditions, which is impractical to verify in numerical practice, we introduce the concepts of truncated Newton step \cite{xie2025remuregionalminimalupdating} and truncated Newton step error as a more practical method to measure the closeness between two models.
In this paper, we primarily consider the twice-differentiable function $h$ as a quadratic model. In this case, it is easy to see when $\Vert(\nabla^2m_k(\bm 0))^{-1}\nabla m_k(\bm 0)\Vert\leq \Delta_k$ (where $\bm 0\in\RR^q$ corresponds to $\bm x_k\in\RR^n$), the truncated Newton step $\bm y=\mathcal{N}(m_k,\bm 0,\Delta_k)$ is exactly the solution of $\nabla m_k(\bm y)=\bm 0$. \cite{xie2025remuregionalminimalupdating}.
\begin{definition}[Truncated Newton step error \cite{xie2025remuregionalminimalupdating}]\label{def: truncated newton step error}
Let $h_1,h_2:\RR^q\to \RR$ be two twice differentiable functions. $\bm y\in\RR^q$ is a point and $\Delta>0$ is trust-region radius. Then the truncated Newton step error is 
\begin{equation*}
    \mathrm{Dist}_{\mathcal{N}}(h_1,h_2,\bm y,\Delta):=\frac{1}{\Delta}\Vert\mathcal{N}(h_1,\bm y,\Delta)-\mathcal{N}(h_2,\bm y,\Delta)\Vert.   
\end{equation*}
\end{definition}
Here we introduce Definition \ref{def: truncated newton step error} of truncated Newton step error primarily to illustrate the difference between gradient-free quadratic models $\hat{m}_k$ (obtained through interpolation) and gradient-based quadratic models $m_k$. We compute the truncated Newton step error $\mathrm{Dist}_{\mathcal{N}}(\hat{m}_k,m_k,\bm 0,\Delta_k)$ for 380 problems (run MD-LAMBO for 1000 iterations) and calculated the ratio of solved problems. The figure and analysis are shown in Section \ref{section: numerical results}.

\section{Convergence of Subspace-Based Trust-Region Method}\label{section: convergence}

In this section, we focus on the convergence result of MD-LAMBO, and we claim that the sequence generated by our algorithm converges to a stationary point under some assumptions. We prove such convergence under the trust-region settings. 
At each iteration, the trust-region method solves a subproblem of the form 
\[
\min_{\bm{y} \in \hat{\mathcal{S}}_k, \ \|\bm{y}\| \leq \Delta_k} m_k(\bm{y}),
\]
where $\mathcal{S}_k$ is a low-dimensional subspace, $\hat{\mathcal{S}}_k$ is its projected counterpart and $m_k$ is the quadratic model introduced in Section \ref{section: Trust region subproblem on subspace}, where $\Delta_k$ is the trust-region radius.
    \begin{assumption}\label{assumption: about f}
        $f:\mathbb{R}^n\to\mathbb{R}$ is bounded from below and twice continuously differentiable, whose Hessian is \textcolor{black}{bounded by $\kappa_H$} and Lipschitz continuous with Lipschitz constant $L$. 
    \end{assumption} 
    Let $\mathcal{K}=\{k\in\mathbb{N}:\rho_k\geq\eta_0\}$. Then $\mathcal{K}$ collects the index of all successful iterations. In fact, we will see through the following lemmas that $\mathcal{K}$ is an infinite index set.  \begin{assumption}\label{assumption: many m_k satisfies cauchy assump}
        For any $k\in\mathcal{K}$, there exists $k'\geq k,\ k'\in\mathcal{K}$ such that $m_{k'}$ satisfies Assumption \ref{assumption: cauchy}.
    \end{assumption}
    Suppose Assumption \ref{assumption: many m_k satisfies cauchy assump} is satisfied for the index set $\mathcal{K}$. Let $\{\bm x_k\}_{k\in\mathbb{N}}$ denote the sequence generated by Algorithm \ref{algo: complete algorithm}. The main result of this section is stated below. \begin{theorem}
     Suppose the object function $f$ satisfies Assumption \ref{assumption: about f}, the subspaces $\mathcal{S}_k$ satisfies $\mathcal{S}_k\neq\mathcal{S}_{\bm s_1,\bm s_2}^{(1,\sigma)}\ (\sigma\geq 0)$, then
       $\inf_{k\leq N}\Vert \nabla f(\bm x_k)\Vert \to 0,\ N\to\infty.$
    \end{theorem}
    The proof of this theorem is split into two lemmas, which discuss two cases separately, where in the first case, model $m_k$ is a good approximation of $f$, i.e. \begin{equation}\label{ineq:convergence_rho_eta}
        \rho_k=\frac{f(\bm{x}_k)-f(\bm{x}_k+\bm{P}_k \bm y^{*})}{m_k(\bm 0)-m_k(\bm y^*)}\geq \eta_0,
    \end{equation}
    and in the second case, model $m_k$ is a poor approximation of $f$, i.e., 
    \begin{equation*}
        \rho_k=\frac{f(\bm{x}_k)-f(\bm{x}_k+\bm{P}_k \bm y^{*})}{m_k(\bm 0)-m_k(\bm y^*)}< \eta_0.
    \end{equation*}
    When the model is a poor approximation, we claim that our algorithm attempts to improve the approximation quality within a finite number of iterations. 
    \begin{lemma}\label{lem: convergence_bad->good}
        Suppose $f$ satisfies Assumption \ref{assumption: about f}. If $\rho_k<\eta_0$ for some $k$, then there exists some positive integer $N>k$ such that $\rho_{N}\geq\eta_0$.
    \end{lemma}
    \begin{proof}{Proof}
        Suppose to the contrary that, starting from the $k$th iteration, all model approximations do not achieve the threshold $\eta_0$, i.e.,  \begin{equation}\label{ineq:convergence_bad_rho}  \rho_t=\frac{f(\bm{x}_t)-f(\bm{x}_t+\bm{P}_t \bm y^{*})}{m_t(\bm 0)-m_t(\bm y^*)}< \eta_0,\text{ for all }t\geq k.
    \end{equation}
    Then by $f$ is twice differentiable at $\bm x_t$, we can write the Taylor expansion of $f$ at $\bm x_t$ with error term $o(\Vert \bm P_t\bm y^*\Vert^2)$, which is
    \begin{equation*}
    \begin{aligned}
        f(\bm{x}_t)-f(\bm{x}_t+\bm{P}_t \bm y^{*})
        &=-\nabla f(\bm x_t)^\top \bm P_t\bm y^*-\frac{1}{2} (\bm P_t\bm y^*)^\top\nabla^2f(\bm x_t)\bm P_t\bm y^*+o(\Vert \bm P_t\bm y^*\Vert^2)\\
        &=m_t(\bm 0)-m_t(\bm y^*)+o(\Vert \bm P_t\bm y^*\Vert^2)\\
        &=m_t(\bm 0)-m_t(\bm y^*)+o(\Vert \bm y^*\Vert^2).
    \end{aligned}   
    \end{equation*}
    Therefore, we can write $\rho_t$ as 
    \begin{equation*}
        \rho_t=1-\frac{o(\Vert\bm y^*\Vert^2)}{m_t(\bm 0)-m_t(\bm y^*)},\text{ for all }t\geq k.
    \end{equation*}
    Moreover, \eqref{ineq:convergence_bad_rho} implies $\bm x_t=\bm x_{k}$ and $\Delta_t=\gamma_{dec}^{t-k}\Delta_{k}$ for all $t\geq k$, together with the fact that
    $m_t(\bm 0)-m_t(\bm y^*)=O(\Vert \bm y^*\Vert)$, we have $\rho_t\to 1 \text{ as }t\to\infty$, which contradicts with \eqref{ineq:convergence_bad_rho}.  
    \end{proof}

    With Lemma \ref{lem: convergence_bad->good}, it suffices to show that for any small $\eps>0$,
    after sufficient iterations with $m_k$ a good model approximation, $\Vert\nabla f(\bm x_k)\Vert$ will be driven under $\eps$.   
    \begin{lemma}
        Suppose Assumption \ref{assumption: many m_k satisfies cauchy assump} is satisfied and $\Delta_k\to0$ as $k\to \infty$, then for any $k\in\mathbb{N},\ \eps>0$, there exists $N>k$ such that $\Vert\nabla f(\bm x_N)\Vert\leq \eps.$  
    \end{lemma}
    
    \begin{proof}{Proof}
       Assumption \ref{assumption: cauchy} guarantees the sufficient decrease of $m_k$. Together with \eqref{ineq:convergence_rho_eta}, we have \begin{equation}\label{ineq:decrease of f in good_approx case}
        f(\bm{x}_k)-f(\bm{x}_k+\bm{P}_k \bm y^{*})\geq \frac{\eta_0}{2}\left(m_k(\bm 0)-m_k(\bm s^*)\right)\geq 0,
    \end{equation}
    where $\bm s^*$ is the minimizer of $m_k$ in the full subspace $\mathbb{R}^q$, and $\bm y^*$ is the solution of the $k$th subproblem, $q=\dim(\mathcal{S}_k)$. Moreover, by the fully quadratic condition, we have
    \begin{equation}\label{ineq:fully quadratic in convergence_good_approx}
        \Vert \bm P_k^\top\nabla f(\bm x_k+\bm P_k\bm y^*)-\nabla m_k(\bm y^*)\Vert\leq \kappa_{eg}\Delta_k^2,
    \end{equation}
    where and $\Delta_k$ is bounded from the above by some positive $\Delta_{\max}$ and by Proposition \ref{prop: m_k fully quadratic}, $\kappa_{eg}=2L$ is independent of $k$.

    For $\eps>0$, if $\Vert\nabla m_{k}(\bm 0)\Vert>\eps$, then by Assumption \ref{assumption: many m_k satisfies cauchy assump}, we can find $k\leq k_0\in\mathcal{K}$, such that solving the trust-region subproblem on $\hat{\mathcal{S}}_{k_0}$ will give a Cauchy decrease of $m_{k_0}$, which is followed by a sufficient decrease in $f$ as described by \eqref{ineq:decrease of f in good_approx case}. Again, if $\Vert\nabla m_{k_0}(\bm 0)\Vert>\eps$, by the same argument, we can find $k_0\leq k_1\in\mathcal{K}$ such that the solution of the $k_1$th trust-region subproblem also gives sufficient function value decrease. Since $f$ is bounded from below, such a decrease can only happen for finite iterations, which implies that for some iteration $k_N\geq k$, there must be $\Vert\nabla m_{k_N}(\bm 0)\Vert\leq \eps$. Hence, the solution of the subproblem at $k_N$th iteration will eventually satisfy $\Vert m_{k_N}(\bm y^*)\Vert\leq\Vert\nabla m_{k_N}(\bm 0)\Vert\leq \eps$. 

    By \eqref{ineq:fully quadratic in convergence_good_approx}, we have \begin{equation}\label{ineq: PkN}
        \Vert \bm P_{k_N}^\top\nabla f(\bm x_{k_N}+\bm P_{k_N}\bm y^*)\Vert\leq \kappa_{eg}\Delta_{k_N}^2+\eps.
    \end{equation}
Moreover, 
\begin{equation*}
\begin{aligned}
     \Vert \bm P_{k_N}^\top\nabla f(\bm x_{k_N}+\bm P_{k_N}\bm y^*)\Vert&\geq \Vert \bm P_{k_N}^\top\nabla f(\bm x_{k_N})\Vert-\Vert \bm P_{k_N}^\top \left(\nabla f(\bm x_{k_N})-\nabla f(\bm x_{k_N}+\bm P_{k_N}\bm y^*)\right) \Vert\\
     &\geq\Vert \bm P_{k_N}^\top\nabla f(\bm x_{k_N})\Vert-\Vert\nabla f(\bm x_{k_N})-\nabla f(\bm x_{k_N}+\bm P_{k_N}\bm y^*)\Vert\\
     &\geq \Vert \bm P_{k_N}^\top\nabla f(\bm x_{k_N})\Vert-\kappa_H\Delta_{k_N}. 
\end{aligned}
\end{equation*}
Since for all truncated subspaces except for $\mathcal{S}_{\bm s_1,\bm s_2}^{(1,\sigma)}\ (\sigma\geq 0)$ we have $\nabla f(\bm x_k)\in \mathcal{S}_k\text{ for all } k$, the equality $\Vert \bm P_{k_N}^\top\nabla f(\bm x_{k_N})\Vert=\Vert\nabla f(\bm x_{k_N})\Vert$ must hold. Combined with \eqref{ineq: PkN}, we have
\begin{equation*}
    \Vert \nabla f(\bm x_{k_N})\Vert=\Vert \bm P_{k_N}^\top\nabla f(\bm x_{k_N})\Vert\leq \kappa_{H}\Delta_k+\kappa_{eg}\Delta_k^2+\eps.
\end{equation*}
Since $\Delta_k\to 0$ as $k\to \infty$, we conclude that for sufficiently large $N$, there must be $\Vert \nabla f(\bm x_{k_N})\Vert\leq 2\eps$.
\end{proof}

\begin{remark}
    When the solution $\bm y^*$ satisfies $\Vert \nabla m_{k}(\bm y^*)\Vert>\Vert\nabla m_{k}(\bm 0)\Vert$, it is reasonable to set $\bm y^*=\bm 0$.
\end{remark}

A key feature of the subspace method is that the algorithm can identify a ``good'' subspace within finitely many steps. Specifically, a subspace is considered ``{good}'' if it is not orthogonal to the gradient $\nabla f$. Such constructions can be justified theoretically using techniques like the {\color{black}Johnson–Lindenstrauss (J-L) transformation \cite{Woodruff_2014} and results. 
In most cases, the subspaces constructed maintain the property of sufficient model decrease. However, specially chosen subspaces can further reduce computational cost at the possible expense of slower functional decrease. 
Therefore, there exists a trade-off: while special subspaces may not guarantee the same rate of function value reduction, they often lead to improved efficiency in large-scale settings.}

\section{Numerical results}\label{section: numerical results}
We use the performance profile (introduced by Dolan and Mor\'{e} \cite{EDD01}) 
to show the numerical behavior of Algorithm \ref{algo: complete algorithm} with different subspaces fitted in, where the problems come from 38 classical test problems\footnote{see https://github.com/POptUS/yatsop.} and 10 forms of variations of each problem (380 problems in total). The specific settings of these test problems can be found in \citet{YATSOp2025}.
{\color{black}Details of the 10 forms of variations can be found in Section 2.3 of \cite{xie2025remuregionalminimalupdating}.
} 

To describe the performance profile, we first define the value
\[f_{\mathrm{acc}}^{N}=\frac{f(\bm{x}_{N})-f(\bm{x}_{0})}{f(\bm{x}_{\text{best}})-f(\bm{x}_{0})} \in [0,1],
\]
where $\bm x_0$ denotes the initial point, $\bm x_N$ is the best point found by Algorithm \ref{algo: complete algorithm} after $N$ iterations 
and $\bm x_{\rm best}$ denotes the best known solution. Then $f_{\rm acc}^N$ reflects the relative decrease of the function value after some number of iterations in a certain algorithm. Let $\tau \in[0,1]$ denote the tolerance. Then the solution obtained by Algorithm \ref{algo: complete algorithm} is said to reach the accuracy $\tau$ when $f_{\rm acc}^{N}\geq 1-\tau$. 
Denote the set of total 380 problems by $\mathcal{P}$ and the set of 16 versions of algorithms by $\mathcal{A}$, where each of the algorithm in $\mathcal{A}$ is obtained by Algorithm \ref{algo: complete algorithm} equipped with the subspaces from Table \ref{tab:subspace-1} and \ref{tab:subspace-2}. Table \ref{tab:algo 1-8} shows the 16 algorithms.

\begin{table}[htbp!]
    \caption{Algorithms equipped with different subspaces}
    \centering
    \begin{tabular}{m{1.2cm}<{\centering}m{1.2cm}
    <{\centering}m{1.2cm}
    <{\centering}m{1.2cm}<{\centering}m{1.2cm}
    <{\centering}m{1.2cm}
    <{\centering}m{1.2cm}
    <{\centering}m{1.2cm}<{\centering}m{1.2cm}<{\centering}}
    \toprule
     Algorithm & 1 & 2 & 3 & 4 & 5 & 6 & 7 & 8 \\
    \midrule
     Subspace  & $\mathcal{S}_{\mathcal{B}}^{(1,0)}$ & $\mathcal{S}_{l}^{(1,0)}$ & $\mathcal{S}_{\bm s_1,\bm s_2}^{(1,0)}$ & $\mathcal{S}_{\langle \bm s_1,\bm s_2\rangle}^{(1,0)}$ &$\mathcal{S}_{\mathcal{B}}^{(1,\sigma)}$ & $\mathcal{S}_{l}^{(1,\sigma)}$ & $\mathcal{S}_{\bm s_1,\bm s_2}^{(1,\sigma)}$ & $\mathcal{S}_{\langle \bm s_1,\bm s_2\rangle}^{(1,\sigma)}$\\
     \midrule
     \midrule
      Algorithm & 9 & 10 & 11 & 12 & 13 & 14 & 15 & 16 \\
    \midrule
      Subspace  & $\mathcal{S}_{\mathcal{B}}^{(2,0)}$ & $\mathcal{S}_{l}^{(2,0)}$ & $\mathcal{S}_{\langle\bm s_1\rangle,\langle\bm s_2\rangle}^{(2,0)}$ & $\mathcal{S}_{\langle \bm s_1,\bm s_2\rangle}^{(2,0)}$ &$\mathcal{S}_{\mathcal{B}}^{(2,\sigma)}$ & $\mathcal{S}_{l}^{(2,\sigma)}$ & $\mathcal{S}_{\langle\bm s_1\rangle,\langle\bm s_2\rangle}^{(2,\sigma)}$ & $\mathcal{S}_{\langle \bm s_1,\bm s_2\rangle}^{(2,\sigma)}$\\
    \bottomrule
    \end{tabular}
    
    \label{tab:algo 1-8}
\end{table}
The relationship of the 16 subspaces is illustrated in Figure \ref{fig: sp_relation}, where $2\xrightarrow[0\ \to\ \sigma]{0\ \supset\  \sigma}6$ means that the \textcolor{black}{6th subspace $\mathcal{S}_l^{(1,\sigma)}$}, while included in the 2nd subspace $\mathcal{S}_l^{(1,0)}$, is driven by the same inspiring region as the 2nd subspace $\mathcal{S}_l^{(1,0)}$ while their models differ only with a penalty term $\sigma \Vert \bm s\Vert ^3/3$. 
Similarly, $\mathcal{S}_{\langle\bm s_1,\bm s_2\rangle}^{(1,0)}\xrightarrow{\subset}\mathcal{S}_{\langle\bm s_1\rangle,\langle\bm s_2\rangle}^{(2,0)}$ means the 4th subspace $\mathcal{S}_{\langle\bm s_1,\bm s_2\rangle}^{(1,0)}$ is included in the 11th subspace $\mathcal{S}_{\langle\bm s_1\rangle,\langle\bm s_2\rangle}^{(2,0)}$ while $\mathcal{S}_{\langle\bm s_1\rangle,\langle\bm s_2\rangle}^{(2,0)}\xrightarrow{=}\mathcal{S}_{\langle\bm s_1,\bm s_2\rangle}^{(2,0)}$ means the 11th subspace $\mathcal{S}_{\langle\bm s_1\rangle,\langle\bm s_2\rangle}^{(2,0)}$ is exactly the same as the 12th subspace $\mathcal{S}_{\langle\bm s_1,\bm s_2\rangle}^{(2,0)}$.

\begin{remark}
The arrow symbol $\rightarrow$ is used to connect two related subspaces, numbers and letters $0, \sigma$ stand for subspaces, and the inclusion and equality symbols $\subset,=$ illustrate the relationship between truncated subspaces. The three types of symbols mentioned are independent of each other. Being combined together, they express the relationship between two different truncated subspaces.
\end{remark}
\begin{figure}[htb]
    \centering   \includegraphics[width=.97\linewidth, trim=50 0 50 0, clip]{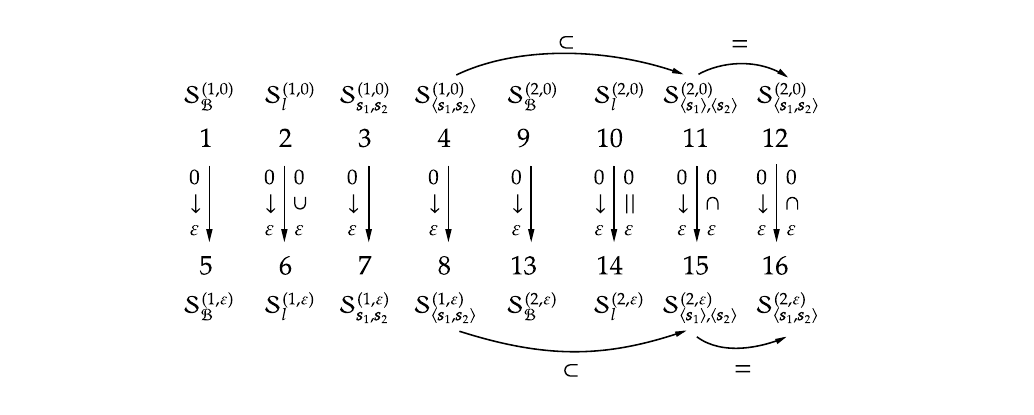}
    \caption{Relationships between our subspaces}
    \label{fig: sp_relation}
\end{figure}

For each $p\in\mathcal{P}$ and $a\in\mathcal{A}$, let $N_{a,p}=\min\{n \in \mathbb{N},\ f_{\mathrm{acc}}^{n}\ge 1-\tau \}$, then $N_{a,p}$ is the minimal number of iterations such that the solution achieves the accuracy $\tau$.

For a fixed value of $\tau$, we define the indicator $T_{a,p}$
\begin{equation*}
\begin{aligned}
&T_{a, p}=\left\{
\begin{aligned}
& 1, \ \text{if} \ f_{\mathrm{acc}}^{N} \geq 1-\tau \ \text{for some } N,\\
& 0, \ \text{otherwise,}
\end{aligned}\right. 
\end{aligned}
\end{equation*}
and the relative number of iterations $r_{a,p}$
\begin{equation*}
r_{a, p}=\left\{\begin{aligned}
& \frac{N_{a, p}}{\min \left\{N_{\tilde{a}, p}:\, \tilde{a} \in \mathcal{A},\ T_{\tilde{a}, p}=1\right\}},\ \text {if} \ T_{a, p}=1, \\
&\ \ \ \ \ \ \ \ \ \ \ \ \ \  +\infty, \ \ \ \ \ \ \ \ \ \ \ \ \ \ \ \ \  \ \  \text {if} \ T_{a, p}=0.
\end{aligned}\right.
\end{equation*}

 In the case where $T_{a,p}=1$, we have $r_{a,p}={\rm Iteration}/{\rm Iteration_{min}}\in [1,+\infty)$, where $\mathrm{Iteration}$ denotes the number of iterations solving problem $p$ with algorithm $a$ and ${\rm Iteration_{min}}$ denotes the minimum number of iterations among all the algorithms in $\mathcal{A}$. For this reason, we call $r_{a,p}$ the relative number of iterations. 
Then, the performance profile is the plot of the map $\rho_a(\alpha)$ with assumed certain accuracy $\tau$, where
\begin{equation*}
    \rho_{a}(\alpha)=\frac{1}{\vert\mathcal{P}\vert}\left\vert\left\{p \in \mathcal{P}: r_{a, p} \leq \alpha\right\}\right\vert.
\end{equation*} 
Fixing $\alpha$, higher value of $\rho_a(\alpha)$ indicates algorithm $a$ solves more problems within $\alpha$ relative number of iterations.

Let $\tau=10^{-1},10^{-2},10^{-3},10^{-4},10^{-5}$ and $10^{-6}$ separately, test the algorithms in $\mathcal{A}$ with the 380 test problems and obtain the Figure \ref{fig:perf_prof} of performance profile.
\begin{figure}[htbp]
    \centering
    \begin{subfigure}[b]{0.48\textwidth}
        \includegraphics[width=\linewidth]{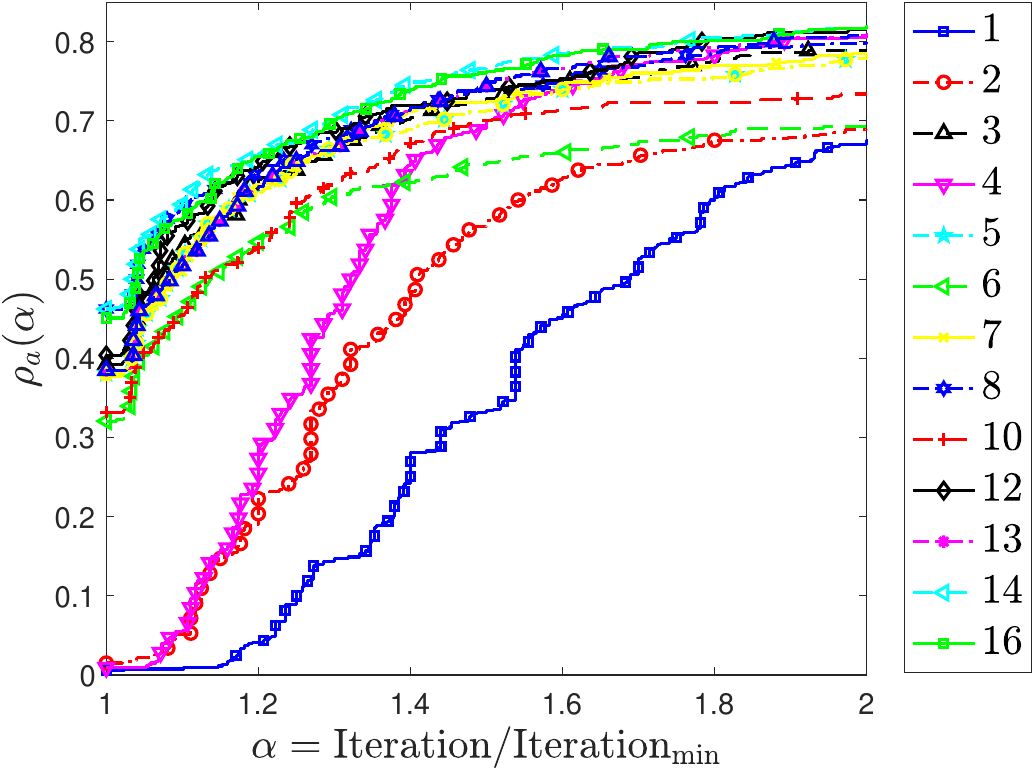}
        \caption{$\tau=10^{-1}$}
    \end{subfigure}
    \begin{subfigure}[b]{0.48\textwidth}
        \includegraphics[width=\linewidth]{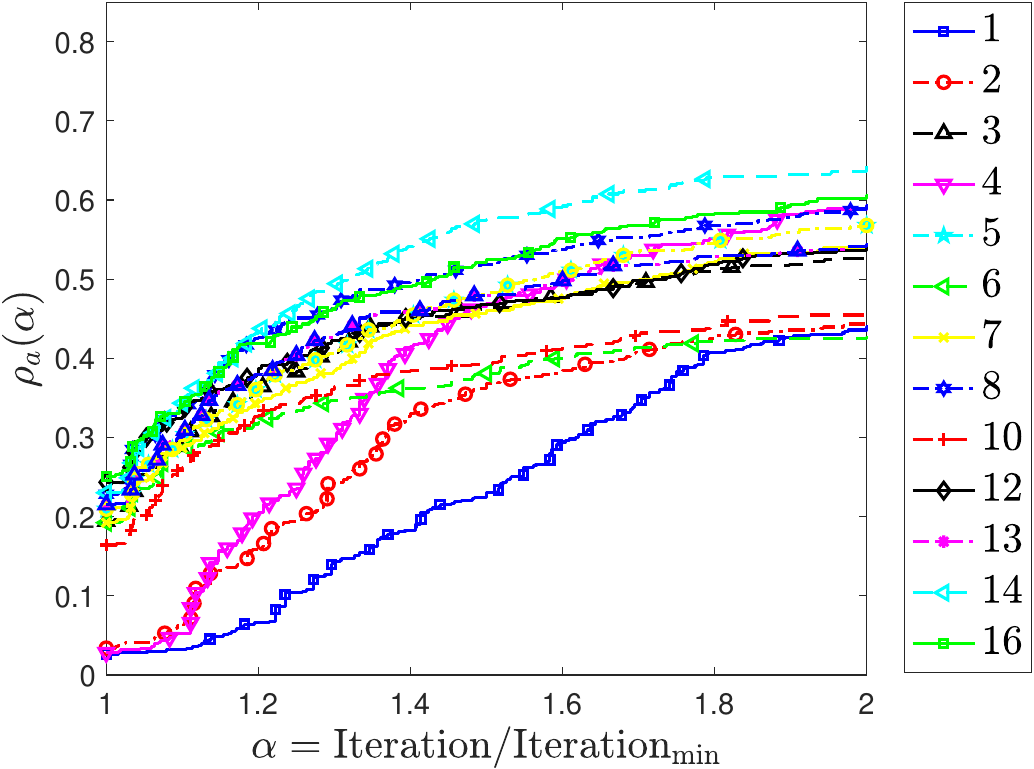}
        \caption{$\tau=10^{-2}$}
    \end{subfigure} \\ 
    \begin{subfigure}[b]{0.48\textwidth}
        \includegraphics[width=\linewidth]{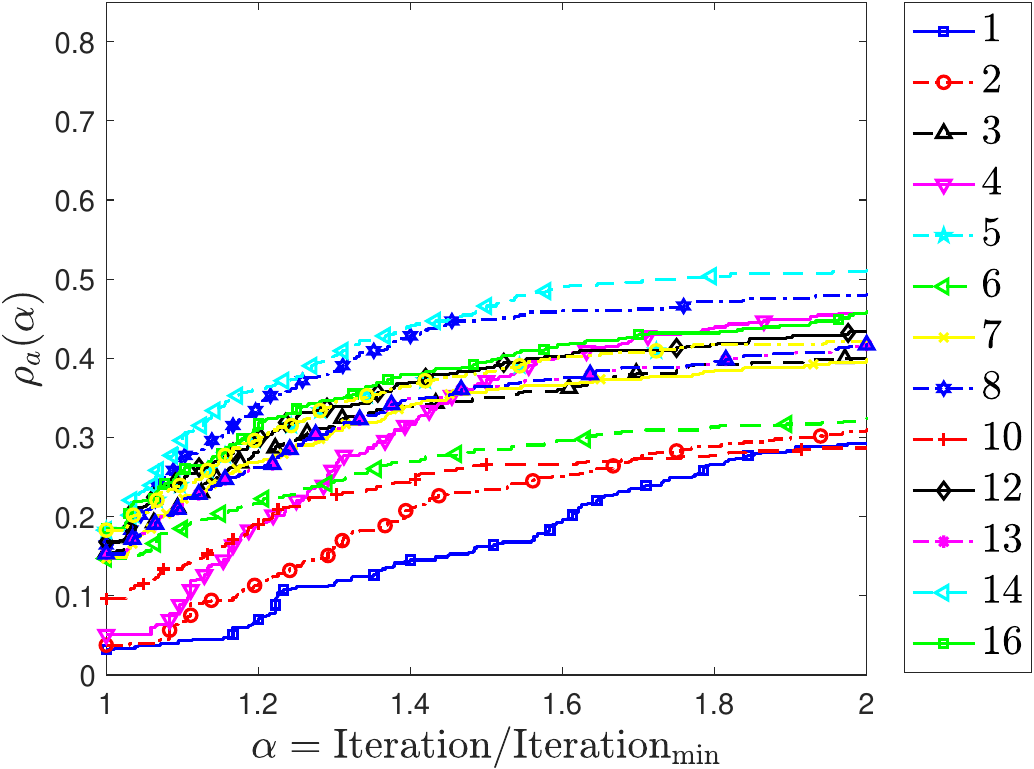}
        \caption{$\tau=10^{-3}$}
    \end{subfigure} 
    \begin{subfigure}[b]{0.48\textwidth}
        \includegraphics[width=\linewidth]{2_perf_profile.pdf}
        \caption{$\tau=10^{-4}$}
    \end{subfigure}\\ 
    \begin{subfigure}[b]{0.48\textwidth}
        \includegraphics[width=\linewidth]{3_perf_profile.pdf}
        \caption{$\tau=10^{-5}$}
    \end{subfigure}
    \begin{subfigure}[b]{0.48\textwidth}
        \includegraphics[width=\linewidth]{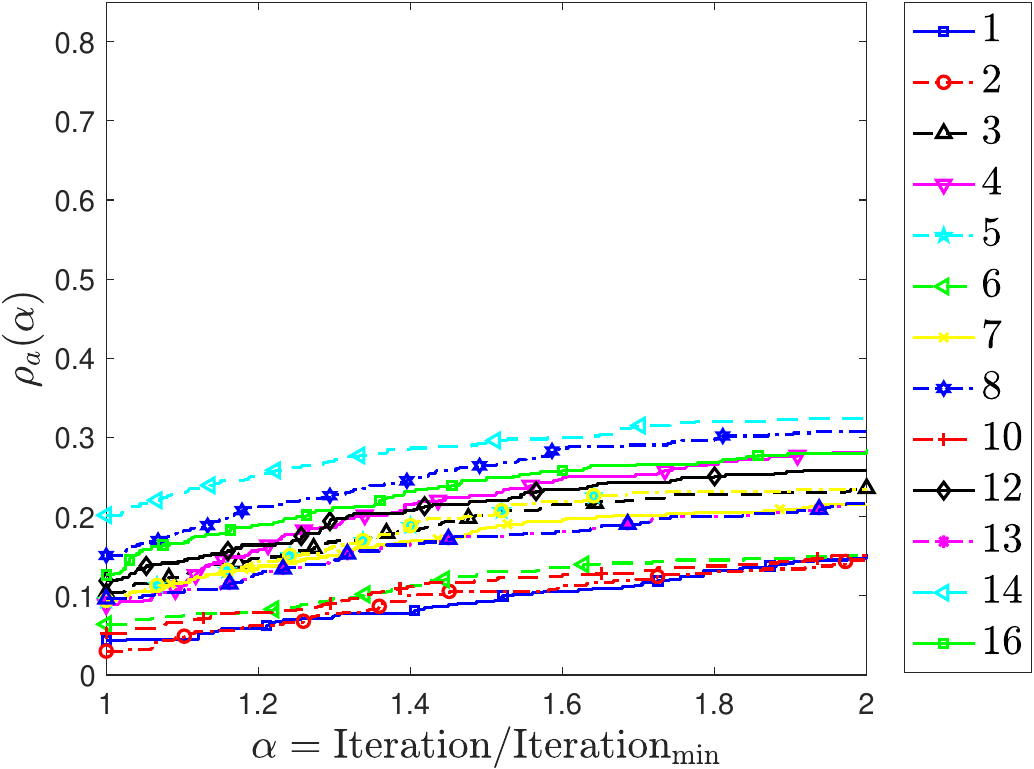}
        \caption{$\tau=10^{-6}$}
    \end{subfigure}
    
    \caption{Performance profile of different subspaces}
    \label{fig:perf_prof}
\end{figure}

Let 
\begin{equation}\label{def:good algos}
  {\rm fast}_p=\{a\in\mathcal{A}:N_{a,p}=\min\{N_{\tilde{a},p}:\tilde{a}\in\mathcal{A},T_{\tilde{a},p}=1\}\}  
\end{equation}
Then ${\rm fast}_p$ is the collection of algorithms that achieve some accuracy with the minimum number of iterations. Algorithms in ${\rm fast}_p$ are considered good algorithms. Moreover, if an algorithm $a\in\mathrm{fast}_p$, we say $a$ is good at solving problem $p$. Therefore, if problem $p$ is solved by an algorithm $a\in\mathcal{A}$ with $\alpha=1$, $a$ is good at solving problem $p$ according to our definition.

A simple observation of Figure \ref{fig:perf_prof} shows that algorithms induced by different subspaces perform differently. In performance profiles, all algorithms tend to solve fewer problems as higher accuracy is required. However, the relative positions of the graphs of $\rho_a(\alpha)$ for different $a\in\mathcal{A}$ are roughly consistent. 
In cases where the accuracy is relatively low ($\tau\geq 10^{-4}$), plots in each figure display two clusters. One cluster is represented by algorithm $1,2,4$
(corresponds to subspaces $\mathcal{S}_{\mathcal{B}}^{(1,0)},\mathcal{S}_{l}^{(1,0)}, \mathcal{S}_{\langle\bm s_1,\bm s_2\rangle}^{(1,0)}$) which has $\rho_a(1)$ close to $0$, implying that such algorithms are not in ${\rm fast}_p$ for almost all $p$. The rest of the algorithms solve positive fractions of test problems when $\alpha=1$, implying that they are good for at least a nonzero fraction of $p\in\mathcal{P}$. Such relatively good algorithms include algorithms induced by subspaces $\mathcal{S}_{l}^{(1,\sigma)}$, $\mathcal{S}_{\langle\bm s_1,\bm s_2\rangle}^{(2,0)}$ where $\mathcal{S}_{l}^{(1,\sigma)}$ is a two-dimensional truncated subspace while $\mathcal{S}_{\langle\bm s_1,\bm s_2\rangle}^{(2,0)}$ takes the classic lower-dimensional subspace ($\cong \mathbb{R}^q$, $q\in\{1,2,3\}$), which allows us to conclude that the truncated subspace has relatively the same capability in problem-solving as the classic subspaces.

Figure \ref{fig:placeholder} displays the distribution of truncated Newton step error of our sixteen subspaces while the error is between the quadratic interpolation model $\hat{m}_k$ and quadratic model $m_k$ defined in Section \ref{section: Trust region subproblem on subspace}. For each subspace (exclude the plot of 11th Subspace $\mathcal{S}_{\langle\bm s_1\rangle,\langle\bm s_2\rangle}^{(2,0)}$, 15th Subspace $\mathcal{S}_{\langle\bm s_1\rangle,\langle\bm s_2\rangle}^{(2,\sigma)}$ and 9th Subspace $\mathcal{S}_{\mathcal{B}}^{(2,0)}$, since we have proved in Section \ref{section: subspace and dimension} that subspace 11 and 12 are the same, subspace 15 and 16 are the same, and both Subspace 9 and 13 equal to $\RR^n$), 380 problems are solved and each within 1000 iterations, among all of the values, those within the first quartile and the third quartile displayed in the figure (represented by the vertical black line). The blue box highlights the fraction of values that lie in the interval $[40\%,60\%]$, capturing the median of the distribution represented by the horizontal line inside. In the figure, all upper bounds 
remain below $2$, which is consistent with the definition. Among all of the distributions, the median and the upper bound of Subspace 13 are approximately the lowest. This is in accordance with the fact that the global interpolation model approximates the object function with a smaller error. Table \ref{tab:classification of subspaces by dim} shows the classification result of Subspaces 1-16 by their dimensions (subspaces inside brackets are omitted in Figure \ref{fig:placeholder}), where the 1st to 8th Subspaces are truncated subspaces while the 9th to 16th Subspaces are classic subspaces.

\begin{table}[htbp]
    \centering
    \caption{Classifications of subspaces by dimensions}
    \centering
    \begin{tabular}{m{3cm}
    <{\centering}m{3cm}
    <{\centering}m{3cm}
    <{\centering}m{3cm}
    <{\centering}}
    \toprule
      Subspaces $\mathcal{S}$ & $\dim(\mathcal{S})\leq 2$ & $3\leq\dim (\mathcal{S})\leq 5$ & $\dim (\mathcal{S})> 5$
     \\
    \midrule
        Serial number  & 2,3,6,7,10,14 & 4,8,12(11),16(15) & 1,5,13(9)\\         
       
    \bottomrule
    \end{tabular}       
    \label{tab:classification of subspaces by dim}
\end{table}
Combining the classification in Table \ref{tab:classification of subspaces by dim} and the Figure \ref{fig:placeholder}, we can see that within the class $\dim(\mathcal{S})\leq 2$, truncated subspaces (Subspace 2, 3, 6, 7, corresponding to  $\mathcal{S}_l^{(1,0)}$, $\mathcal{S}_{\bm s_1,\bm s_2}^{(1,0)}$, $\mathcal{S}_l^{(1,\sigma)}$, $\mathcal{S}_{\bm s_1,\bm s_2}^{(1,\sigma)}$) shows generally lower medians in comparison with classic subspaces representing by 10th Subspace $\mathcal{S}_{l}^{(2,0)}$ and 14th Subspace $\mathcal{S}_l^{(2,\sigma)}$. Moreover, error distributions of these low-dimensional truncated subspaces show a similar level of boxes and medians as high-dimensional subspaces (both truncated and classic, in class $\dim(\mathcal{S})>5)$. Observations above imply that interpolation models can be safely used under the derivative-free case when the subspace is chosen to be a low-dimensional truncated subspace.

\begin{figure}[htbp]
    \centering
    \includegraphics[width=.98\linewidth]{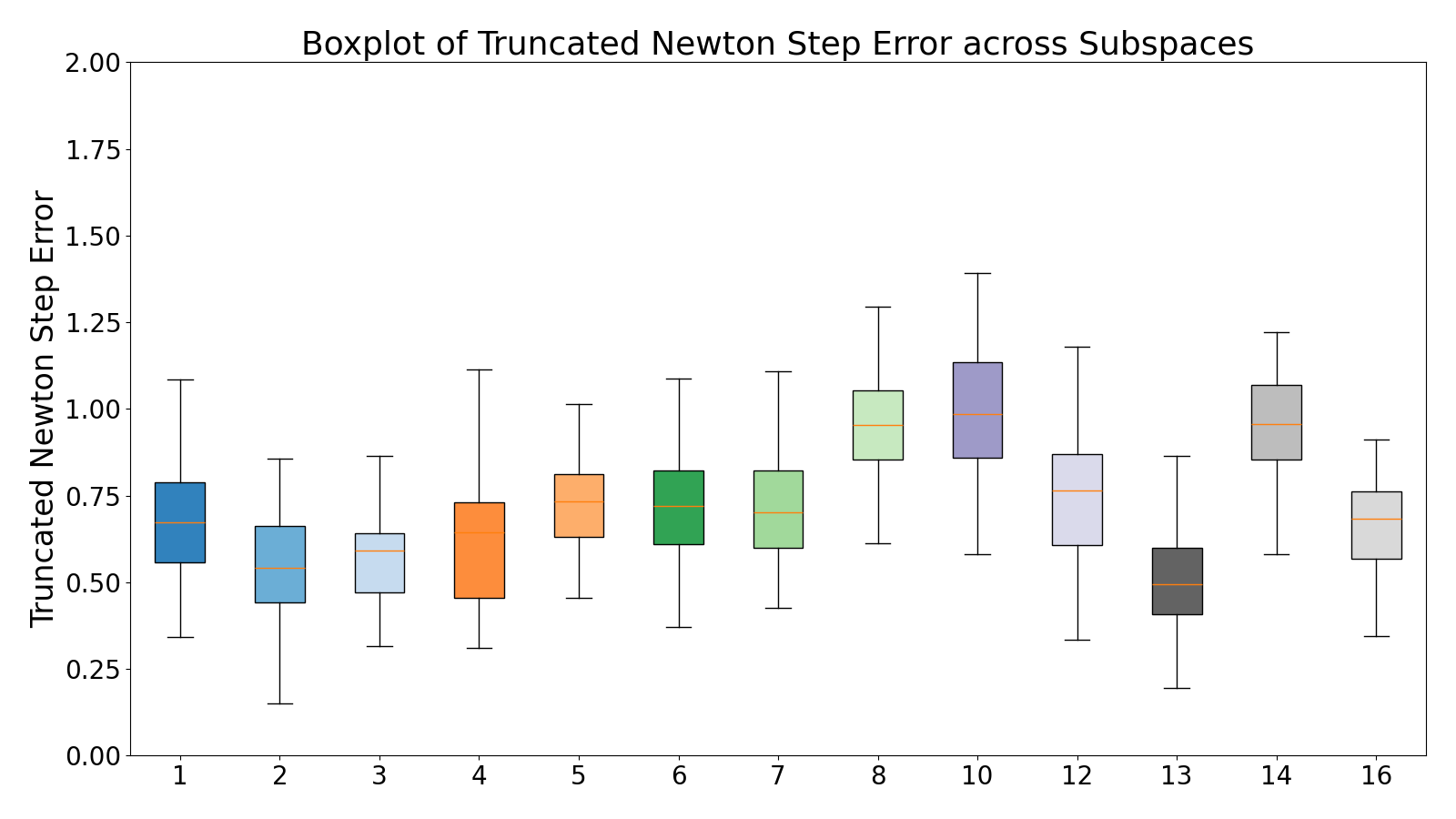}
    \caption{Truncated Newton Step Error}
    \label{fig:placeholder}
\end{figure}

Finally, we present the performance profile of MD-LAMBO (under 6th Subspace $\mathcal{S}_l^{(1,\sigma)}$), SGD, and L-BFGS,  in Figure \ref{fig: Performance profile of different algorithms} to draw a picture of the efficiency of our algorithm. The classical stochastic gradient descent (SGD) method \cite{Ghadimi2013} is used for comparison, where the step size was initialized as $\eta = 10^{-3}$ and decreased by a factor of $0.1$ every $50$ iterations, a momentum term of $0.9$ was employed to accelerate convergence, 
and a damping factor (weight decay) of $10^{-4}$ was applied to improve stability. The method was run in a stochastic setting with uniformly sampled data points at each iteration. We also implemented the limited-memory BFGS algorithm \cite{Liu1989OnTL}, 
which maintains a memory of the most recent \textcolor{black}{$20$} curvature updates to approximate the inverse Hessian. We include these two algorithms as baselines primarily because L-BFGS is famous for its fast convergence and efficiency in solving large problems, and SGD is one of the most widely used algorithms in practical applications. Notice that for all three algorithms involved, a convergence tolerance of $10^{-6}$ on the gradient norm was used. Figure \ref{fig: Performance profile of different algorithms} shows that, as the number of (relative) iterations increases, our algorithm shows the most pronounced improvement in the proportion of problems solved among the three algorithms, especially when the accuracy is relatively low ($\tau=10^{-2}$). In contrast, the plot corresponding to L-BFGS remains essentially flat, showing almost no variations with respect to iterations. By examining the intersections of the three curves with the vertical axis at $\alpha=1$, it can be observed that when the accuracy requirement is sufficiently high ($\tau\leq 10^{-4}$), plot corresponds to our algorithm intersects with the axis at the highest point, followed by plot of SGD and then L-BFGS, reflecting that our algorithm is good at solving the most problems, followed by SGD and then L-BFGS. As the accuracy varies from $10^{-2},10^{-4}$ to $10^{-6}$, the plot of our algorithm rises above the other two plots, which implies that MD-LAMBO (under 6th Subspace $\mathcal{S}_l^{(1,\sigma)}$) has greater capability in problem-solving in comparison to both the average level represented by SGD and an high-efficiency choice represented by L-BFGS. This aligns with our analysis of Algorithm \ref{algo:tcg with radius, steps form} in Section \ref{section: Trust region subproblem on subspace}.

\begin{figure}[htbp]
    \centering
    \begin{subfigure}[b]{0.28\textwidth}
        \includegraphics[height=0.96\linewidth, trim=0 0 130 0, clip]{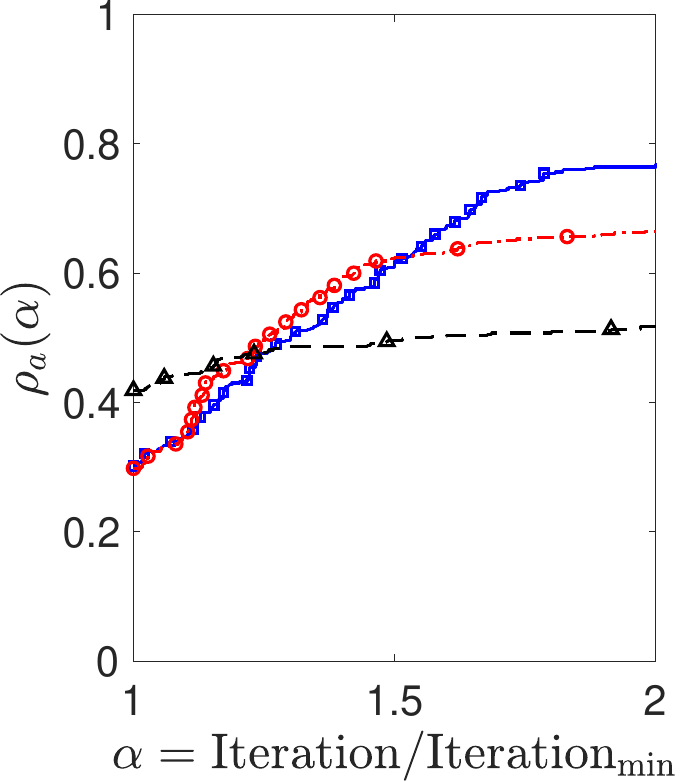}
        \caption{$\tau=10^{-2}$}
    \end{subfigure}  
    \begin{subfigure}[b]{0.28\textwidth}
        \includegraphics[height=0.96\linewidth, trim=0 0 130 0, clip]{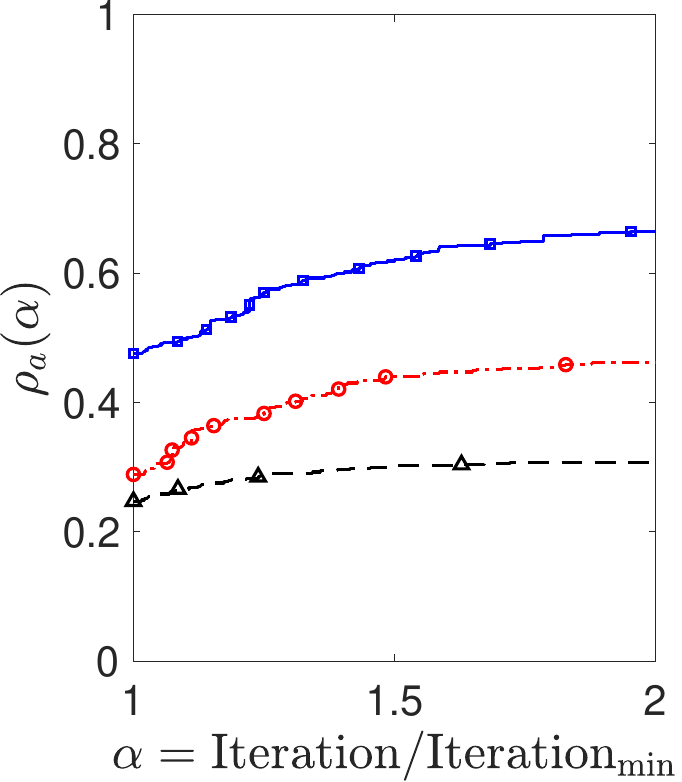}
        \caption{$\tau=10^{-4}$}
    \end{subfigure}
    \begin{subfigure}[b]{0.28\textwidth}
        \includegraphics[height=0.96\linewidth, trim=0 0 0 0, clip]{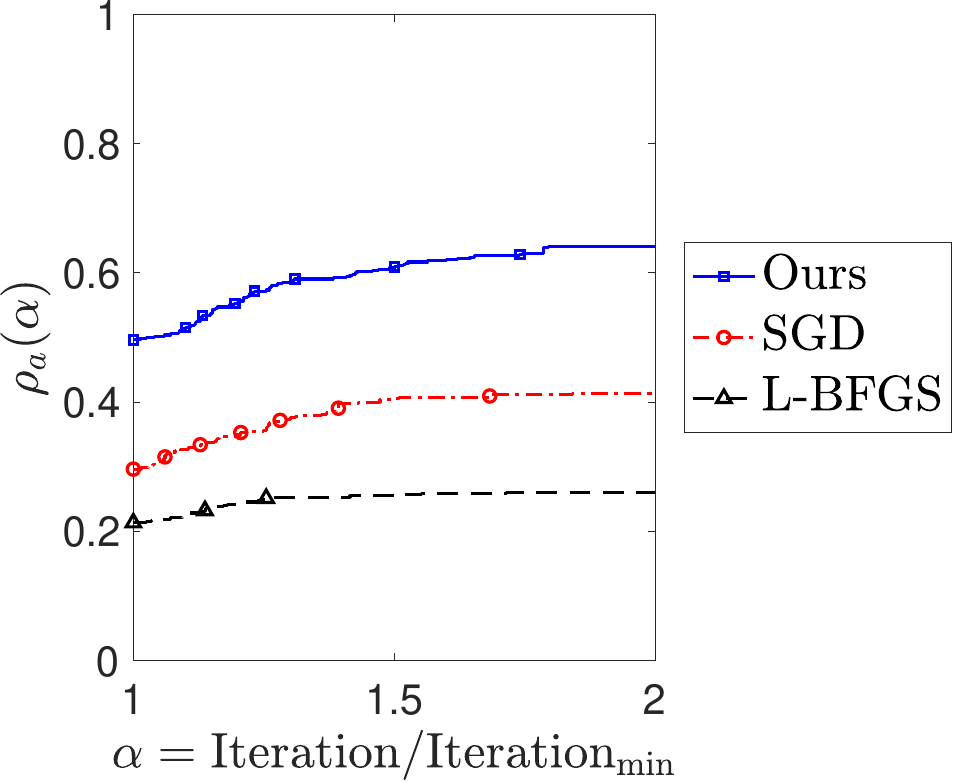}
        \caption{$\tau=10^{-6}$}
    \end{subfigure}
    
    \caption{Performance profile of different algorithms}
    \label{fig: Performance profile of different algorithms} 
\end{figure}

\section{Conclusion}
{\color{black}
This paper addresses the critical challenge of solving large-scale optimization problems by investigating advanced subspace-based methods that leverage local approximation strategies. By introducing and analyzing novel types of subspaces, this paper establishes theoretical guarantees of sufficient function value decrease. Numerical experiments further validate the effectiveness of the proposed subspaces, demonstrating their advantages over traditional approaches in driving optimization efficiency and performance.
}


\bibliographystyle{plainnat}
\bibliography{sample}

\end{document}